\documentclass[12pt]{amsart}
\usepackage{amsfonts,amsmath,amssymb,amsthm,amscd}
\usepackage{graphicx}

\evensidemargin = \oddsidemargin
\textheight = \paperheight
\advance\textheight by -2in
\textwidth = \paperwidth
\advance\textwidth by -2in

\makeatletter
\gdef\thmhead@plain#1#2#3{%
  \thmname{#1}\thmnumber{\@ifnotempty{#1}{ }#2}%
  \thmnote{ {\mdseries#3}}}
\let\thmhead\thmhead@plain
\makeatother
\theoremstyle{plain}

\newtheorem*{thmii}{Theorem}
\newtheorem*{conj}{Conjecture}
\newtheorem{thm}{Theorem}[section]
\newtheorem{lem}[thm]{Lemma}

\newtheorem{prop}[thm]{Proposition}
\newtheorem{crr}[thm]{Corollary}

\theoremstyle{definition}
\newtheorem{dfn}[thm]{Definition}

\newtheorem{rmk}[thm]{Remark}

\newtheorem{assu}{Assumption}

\theoremstyle{remark}

\def\alinea#1{\hfill\break%
  \hbox to \parindent{\hss{\upshape{\bf #1)}}\enspace}\ignorespaces}
\def\bul{\hfill\break\hbox to\parindent{\hss$\bullet$\enspace}\ignorespaces}

\def\mod{\operatorname{mod}}

\def\l{\lambda}
\def\C{\mathbf{C}}

\def\D{\mathbf{D}}

\def\N{\mathbf{N}}
\def\M{\mathbf{M}}

\def\R{\mathbf{R}}
\def\S{\mathbf{S}}

\def\Z{\mathbf{Z}}

\def\CC{\mathcal{C}}

\def\HH{\mathcal{H}}

\def\MM{\mathcal{M}}
\def\NN{\mathcal{N}}

\def\Cap_#1{\bigcap\limits_{#1}}
\def\Cup_#1{\bigcup\limits_{#1}}
\def\ol{\overline}

\newcommand{\dbas}{{\circ \! \circ}}

\makeatletter
\def\cqfdsymb{\relax\protect\ifmmode\else\unskip\nobreak\fi
\quad\hfill$\bgroup
\vcenter{\hrule\hbox{\vrule\@height.6em\kern.6em\vrule}\hrule}\egroup$}
\def\cqfd{\cqfdsymb\endtrivlist}
\gdef\rom#1{\leavevmode\skip@\lastskip\unskip\/%
        \ifdim\skip@=\z@\else\hskip\skip@\fi{\normalshape#1}}
\makeatother

\begin{document}

\title{Newton maps as matings of cubic polynomials}
\author{\sc Magnus Aspenberg and Pascale Roesch}

\begin{abstract}
In this paper we prove existence of matings between a large class of renormalizable cubic polynomials with one fixed critical point and another cubic polynomial having  two fixed critical points. The resulting mating is a Newton map. Our result is the first part towards a conjecture by Tan Lei, stating that all (cubic) Newton maps can be described as matings or captures.
\end{abstract}

\date{\today}
\maketitle

\section{Introduction}

The notion of {\it matings} was introduced in~\cite{Do} as a way to partially parameterize
the space of rational maps of a degree $d \geq 2$ with pairs of polynomials of the same degree $d$. Roughly speaking, the construction
 is to glue the (supposedly locally connected) filled Julia sets $K_1$ anf $K_2$ of a pair of polynomials $f_1$ and $f_2$  along
 their boundaries in reverse order. If no topological obstructions occur
 the resulting set is homeomorphic to the sphere, where $f_1$ and $f_2$ induces a new map $f_1 \uplus_F f_2$ from the sphere to itself.
 This map would then be the (topological) mating of $f_1$ and $f_2$. If one can turn this map into a rational map with a homeomorphic
 change of variables,  then we speak of a conformal mating of $f_1$ and $f_2$. The precise definitions follow.

Our  paper is, to a large part, motivated by such a description  of rational maps of degree~$3$ and a paper by Tan Lei \cite{TL},
where she studied cubic Newton maps. Cubic Newton maps are maps of the form $$N(z)=z-\frac{P(z)}{P'(z)},$$
where $P$ is a cubic polynomial. In \cite{TL} Tan Lei gave a full description of post critically finite Newton maps of degree $3$ in terms of matings and captures. In the same paper she conjectured  that the set of all cubic Newton maps can be completely described in terms of matings and captures. Our paper  answers her conjecture for a large class of maps which are neither  post-critically finite nor hyperbolic, namely when the map admits a
{\it quadratic-like} restriction around its free critical point, hence is {\it renormalizable}.

The study of the remaining maps in the cubic family and (corresponding) Newton maps not covered in this paper is planned in a forthcoming paper.
 Combining these results with L. Tan's result, we hope to describe all Newton maps with locally connected Julia set as matings or captures.

Several works on mating polynomials have been done in degree $2$. Let us recall some related facts for degree two maps. Douady and
Hubbard stated the following conjecture.
\begin{conj}
The points $c_1$ and $c_2$
do not lie in conjugate limbs of the Mandelbrot set if and only if $f_{c_1}(z) = z^2 + c_1$ and $f_{c_2}(z) = z^2 + c_2$ are (conformally)
mateable.
\end{conj}

The post-critically finite case is settled by works of L. Tan, M. Rees and M. Shishikura \cite{TL2} and \cite{Shi}.
By quasi-conformal deformations, matings between hyperbolic polynomials follow from these works.
Concerning post-critically infinite non-hyperbolic matings, M. Yampolsky and S. Zakeri \cite{YZ} showed the existence of matings
between Siegel quadratic polynomials, where the rotation number is of bounded type. Several works have been done in the family $V_2$,
being rational maps having a fixed period $2$ super-attracting cycle. This family describes matings between the so called ``star-like'' basilica polynomial $f(z) = z^2-1$ and other polynomials
not in the $1/2$-limb of the Mandelbrot set, see e.g. \cite{AY}, \cite{Timorin}, \cite{Dudko}. The family $V_3$,
being a family of maps having a fixed period $3$ super-attracting cycle, seems a lot more complicated
(see e.g. \cite{ReesV3}).

For higher degree, most relevant for this paper is the work by  Tan Lei, which concerns the post-critically finite case (see also \cite{Mimat}, \cite{Shi2}).
She gives  in this case the  complete following description.
 \begin{thm}[(Tan Lei)]
There is a set $A$ of cubic polynomials, and a subset $Y$ of the filled Julia set of $f_{\dbas}(z)= z (z^2 + 3/2)$ and a surjective
mapping $M$ onto the set of postcritically finite cubic Newton maps such that: for $g \in A$, the map $M(g)$ is Thurston equivalent
to the mating $f \uplus_F g$ and for $y \in Y$ the map $M(y)$ is a capture.
\end{thm}

\begin{figure}
  \begin{center}
  \includegraphics[scale=0.1]{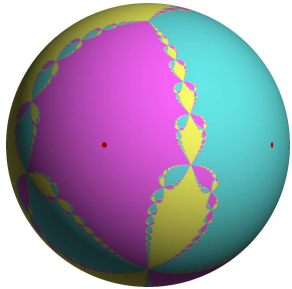}
  \end{center}
  \caption{The Julia set of  a cubic Newton map.}
  \label{Newton map}
\end{figure}

Capture components are not matings, but rather components where the free critical point lies in the basin of attraction of a super-attracting cycle. We will not discuss them in this paper.
In connection to the above result, Tan Lei's conjectured:
\begin{conj}
The fundamental part of the cubic Newton family is homeomorphic to the quotient of a well determined subset of the $a$-family union a specific subset of the filled Julia set of $f_{\dbas}$, by the equivalence relation generated by external rays.
\end{conj}

In other words, conjecturally, every cubic Newton map is, up to affine conjugacy, either a mating between the double-basilica $f_{\dbas}$ and some $f_a$ or a capture.

The main novelty (and difficulty)  in our work  with respect to Tan Lei's work is  that the maps we consider are neither
post-critically finite nor hyperbolic maps (they can be obtained by quasi-conformal deformations).
Our result is aimed as a first step towards Tan Lei's conjecture.

Let us  recall now some definitions to be able to state precisely our theorem.

\subsection*{Acknowledgements}
We thank Tan Lei for leading us into this project. We gratefully acknowledge funding from ANR Grant No. ANR-13-BS01-0002
and Folke Lanner's Fond.

\subsection{Definition of mating}  An excellent introduction  to matings can be found in~\cite{Mimat}.
There are several definitions of matings. The most commonly used definition is purely topological. However, for  our purpose, the most useful is the one  introduced by Yampolsky and Zakeri in \cite{YZ} (also used in \cite{AY}).
 We recall them  now.

 Let $f_1,f_2$ be two monic polynomials of the same degree $d$. Denote by $K_j$ the filled-in Julia set of $f_j$, {\it i.e.} $K_j:=\{z\in \C\mid f_j^n(z)\nrightarrow \infty\}$.  We suppose $K_j$ connected and locally connected. Consequently  the complement is conformally isomorphic to the complement of the disk. We can choose   this conformal map $\Phi_j:\C\setminus \overline \D \to \C \setminus K_j$   tangent to identity at infinity and define   rays  in $  \C \setminus K_j$ as the images  of $\{re^{i2\pi t}\mid r\ge 1\}$. We denote these rays by $R_j(t)$. Note that  this conformal map $\Phi_j$ extends continuously   to the boundary ({\it i.e.} to $\mathbb S^1$)
 by   Carath\'eodory's Theorem.

  Let $S^2$ be the unit sphere in $\C\times\R$.
 Identify each complex plane $\C$ containing $K_i$ (dynamical plane of $f_i$), with the northern hemisphere $\mathbb{H}_+$ for $i=1$ and southern hemisphere $\mathbb{H}_-$ for $i=2$, via the gnomic projections,
\[
\nu_1: \C \rightarrow \mathbb{H}_+ \qquad \nu_2:\C \rightarrow \mathbb{H}_- ,
\]
where $\nu_1(z) = (z,1)/\sqrt{|z|^2+1}$ and $\nu_2(z) = (\bar{z},-1)/\sqrt{|z|^2+1}$. This makes $\nu_2$ equal to $\nu_1$ composed with a 180 degree rotation around the $x$-axis.

It is now not hard to check that the ray $\nu_1(R_1(t))$ of angle $t$ in the northern hemisphere land at the point $(e^{2\pi i t},0)$ on the equator (the unit circle in the plane between the hemispheres). Similarly the ray $\nu_2(R_2(-t))$ on the southern hemisphere of angle $-t$ lands at the point $(e^{2\pi i t},0)$ also. The functions $\nu_i \circ f_i \circ \nu_i^{-1}$ from one hemisphere onto itself are well defined. Moreover, if we approach the equator along the two rays with angle $t$ and $-t$ respectively, both maps $\nu_1 \circ f_1 \circ \nu_1^{-1}$ and $\nu_2 \circ f_2 \circ \nu_2^{-1}$ are going to converge to the same map $(z,0) \rightarrow (z^2,0)$ on the equator. Hence we can glue the two maps together along the equator to form a well defined smooth map  from $S^2$ onto itself. This map, denoted by $f_1 \uplus f_2$ is called the {\em formal mating} of $f_1$ and $f_2$.

\begin{figure}[h]
  \begin{center}
  \includegraphics[scale=0.2]{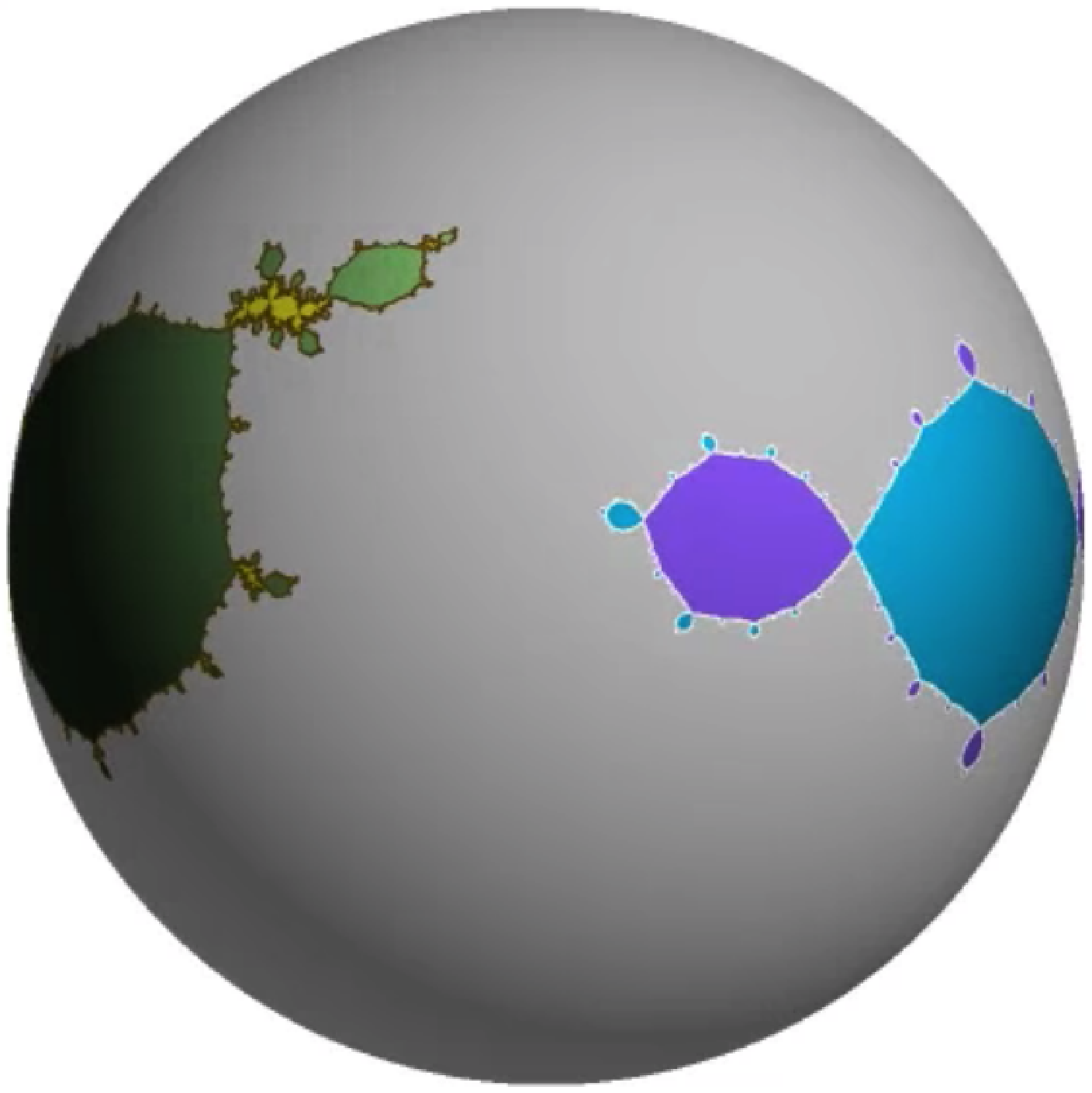}
 \includegraphics[scale=0.2]{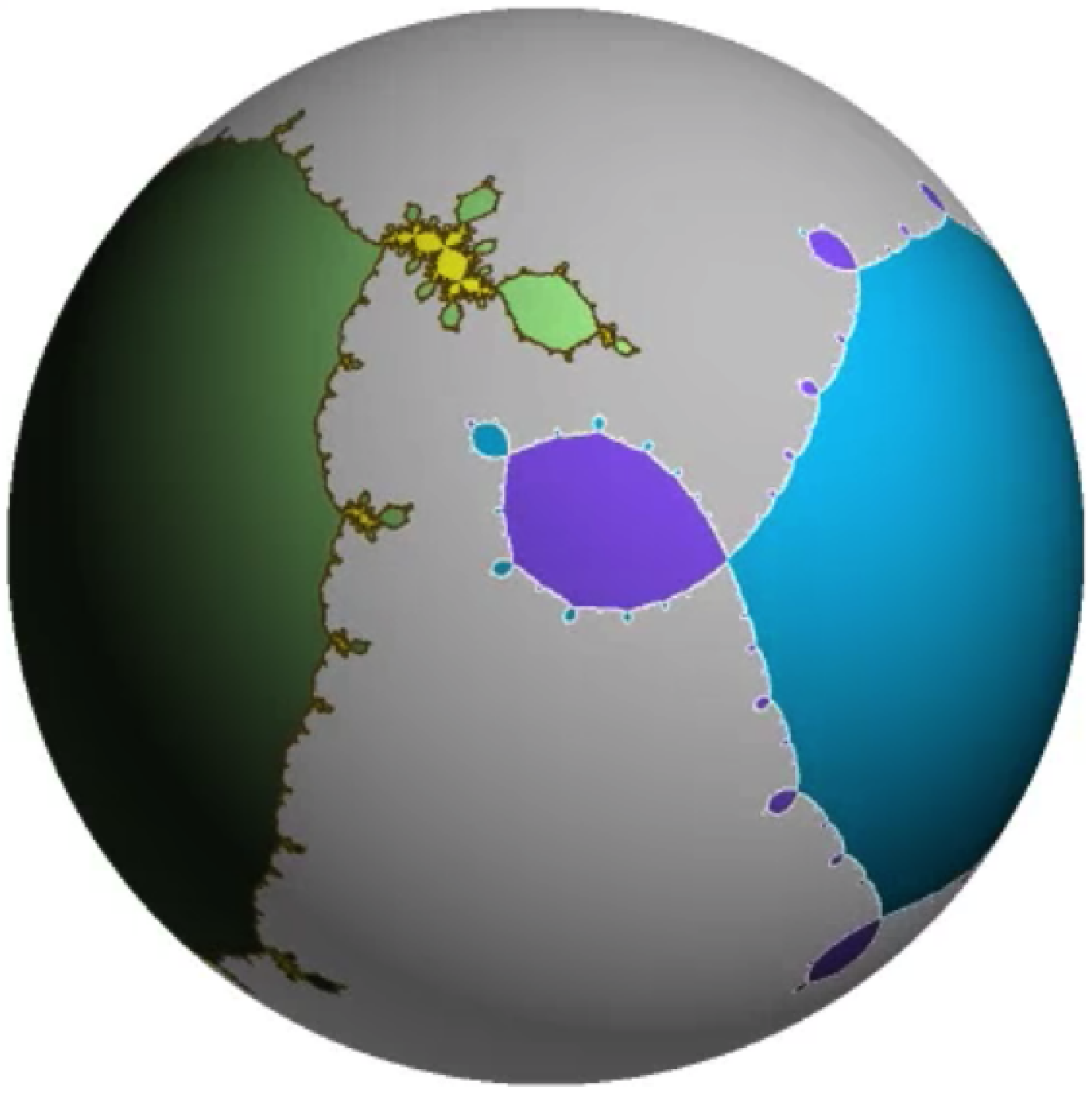}
\includegraphics[scale=0.2]{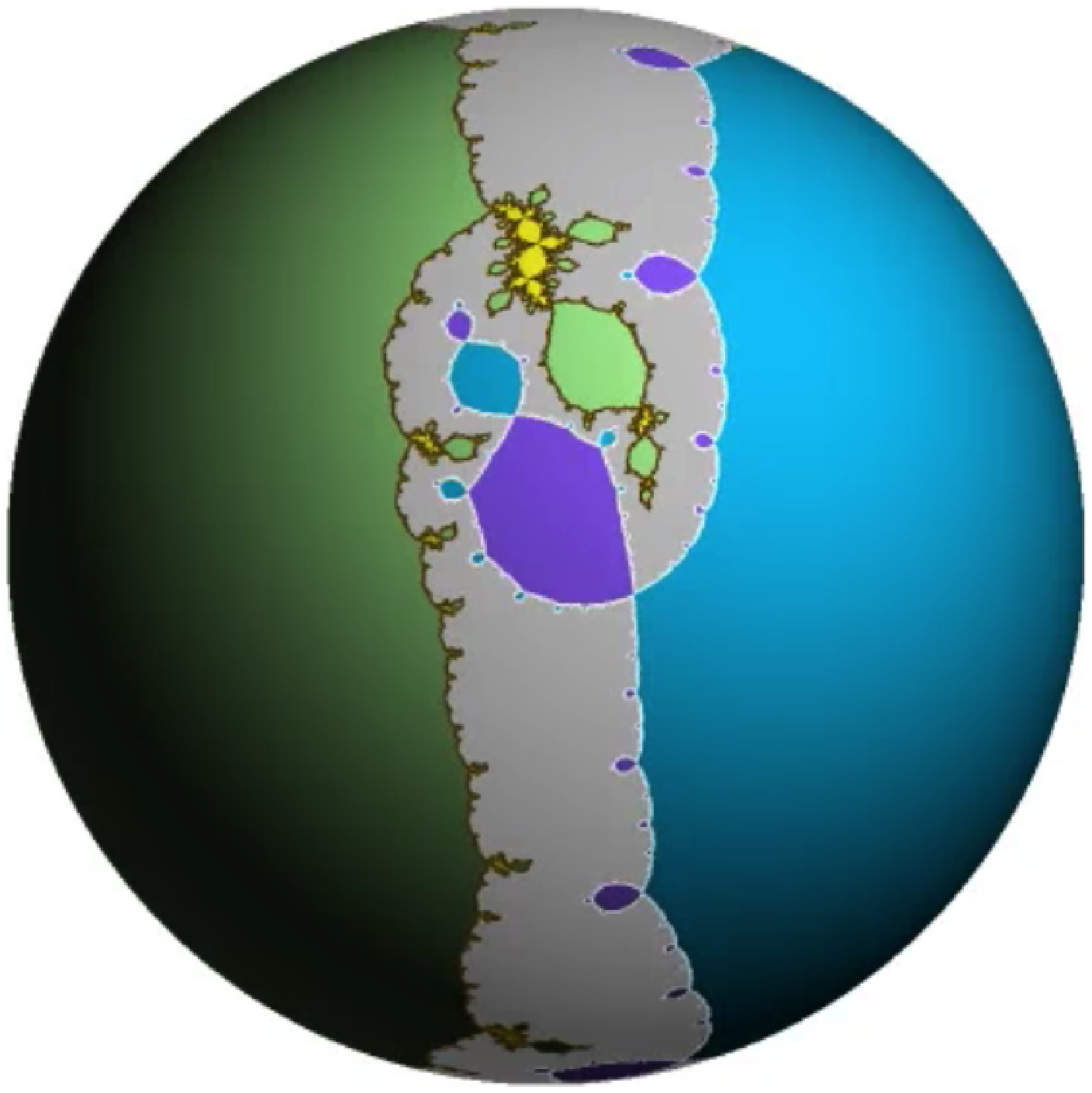}
\includegraphics[scale=0.2]{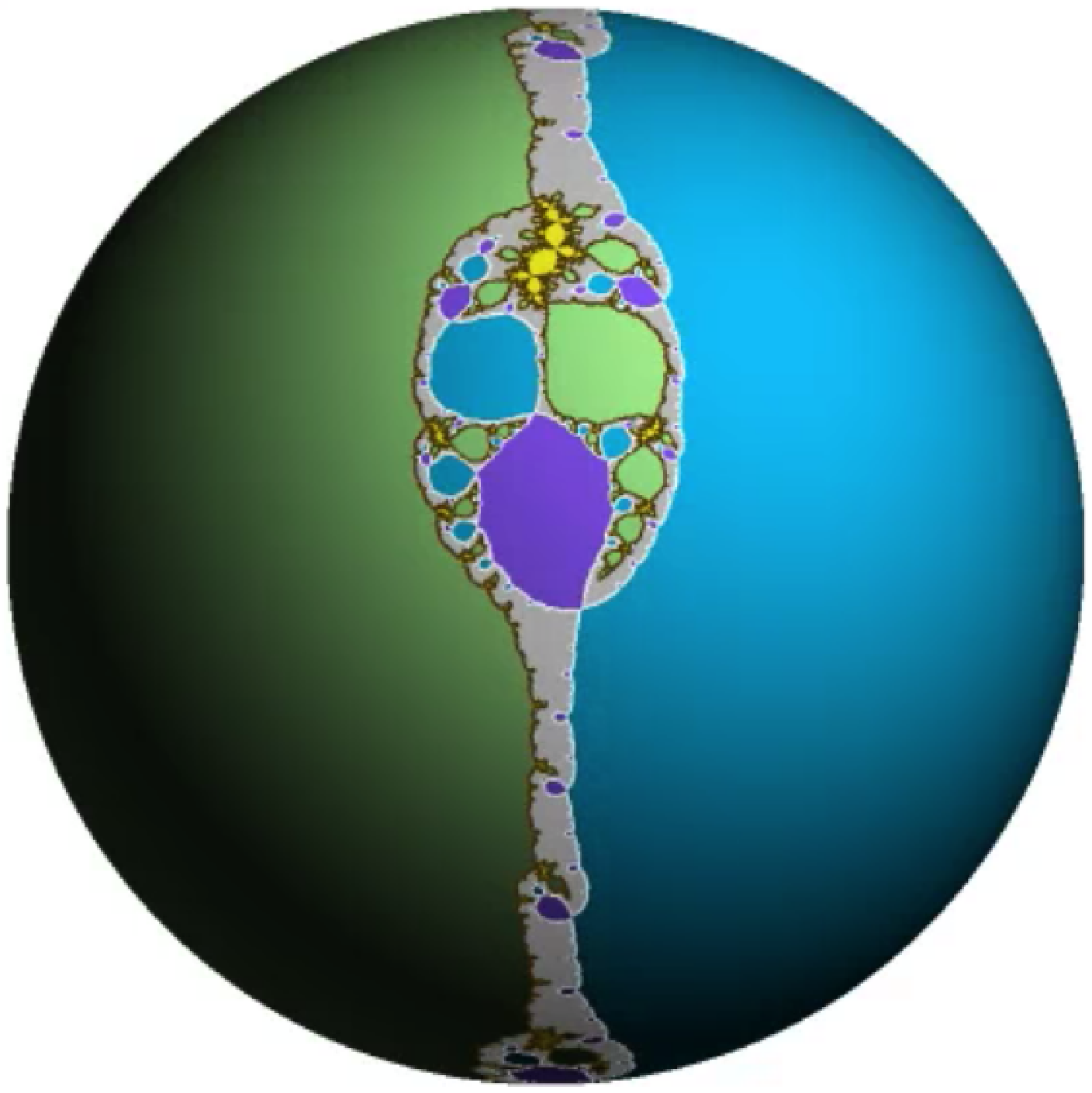}
\includegraphics[scale=0.2]{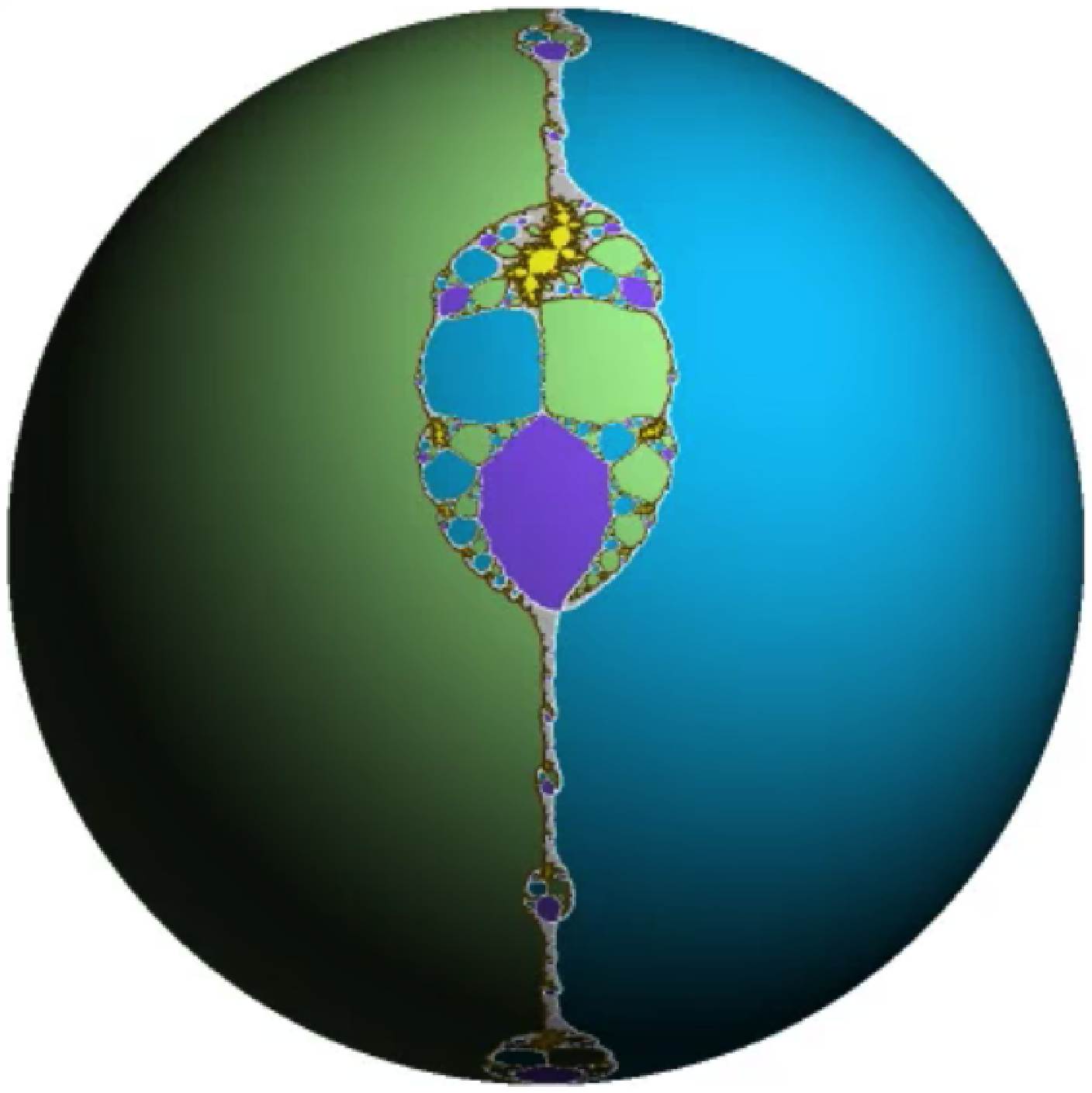}  \end{center}
  \caption{Sketch of the mating (images courtesy of A. Ch\'eritat). }
    \label{Mating movie}
\end{figure}

Define the ray equivalence relation to be the smallest
equivalence relation $\sim_r$ on $S^2$  such that the closure
of the image $\nu_1(R_1(t))$, as well as the closure of
$\nu_2(R_2(-t))$ lies in a single equivalence class.
The map induced on the quotient $S^2 / \sim_r$  by  $f_1 \uplus f_2$
is called the {\it topological mating} of $f_1$ and $f_2$
and denoted by $f_1 \coprod f_2$.

Note however, that this is not the standard way to define topological mating. It is in general defined as a map acting on the space obtained by  gluing the filled in Julia sets $K_i$ along their boundaries.

Suppose now that $S^2 / \sim_r$ is homeomorphic to the sphere $S^2$ and denote by  $\pi_F: S^2 \rightarrow S^2 / \sim_r$ the natural projection.  We say that $f_1$ and $f_2$ are {\em conformally mateable} if there exist a homeomorphism $h$ and a rational map $R$ such that the following diagram commutes
\[
\begin{CD}
S^2    @>     f_1 \coprod f_2 >>    S^2 \\
@V h VV                @VV h V \\
\hat{\C}   @>     R>>   \hat{\C}
\end{CD}
\]
and such that the maps $h \circ \pi_F \circ \nu_j $ are holomorphic on the interior of $K_j$ ( the maps $ \pi_F \circ \nu_j$ are complex charts for $S^2$).
If such $R$ is unique up to M\"obius conjugacy, we refer to it as {\it the mating of $f_1$ and $f_2$}.

The presentation of the topological mating through the formal mating has the advantage that we can make use of Moore's theorem (see~\cite{Moore}); this theorem gives a criterium on the equivalence relation to get a topological  sphere as quotient.

\begin{thm}[(R.L. Moore)]
Let $\sim$ be any topologically closed equivalence relation on $S^2$,
with more than one equivalence class and with only connected
equivalence classes. Then $S^2 / \sim$ is homeomorphic to $S^2$ if and
only if each equivalence class is non separating.
Moreover let $\pi: S^2 \rightarrow S^2 / \sim $ denote the natural projection. In the positive case above we may choose the homeomorphism
$h: S^2 / \sim \ \rightarrow  S^2$ such that the composite map $h \circ \pi$ is a uniform limit of homeomorphisms.
\end{thm}

We now give an equivalent definition of conformal mating, which seems to originate from Hubbard and    used by   Yampolsky-Zakeri (see also~\cite{MP} for more details). We will adopt  this definition in the present paper.
\begin{dfn} The two polynomials $f_1$ and $f_2$ are said {\em conformally mateable}, or just mateable, if
 there exist a rational map $R$ and  two semi-conjugacies $\phi_j: K_j \rightarrow \hat{\C}$ conformal on the interior of $K_j$, such that  $\phi_1(K_1)\cup \phi_2(K_2)=\hat{\C}$ and  \[ \forall (z,w)\in K_i\times K_j, \quad \phi_i(z)= \phi_j(w) \iff z \sim_r w.\] 

The rational map  $R$ is called  a {\em conformal mating}. Moreover, $R$ is topologically conjugate to the topological mating.
\end{dfn}

A semi-conjugacy is a continuous map satisfying a conjugacy relation without being necessarily injective.

\subsection{Statement of results}
A {\it cubic Newton map} is a rational map of degree $3$ of the form $$N(z)=z-\frac{P(z)}{P'(z)}$$ where $P$ is a cubic polynomial. The roots of $P$ should be distinct and are critical fixed points. Therefore, if $N$ arises as a mating of two polynomials of degree $3$, one polynomial will  have two critical fixed points and the other one should have at least one.

\begin{figure}[ht]
 \begin{center}
 \includegraphics[scale=0.4, angle=90]{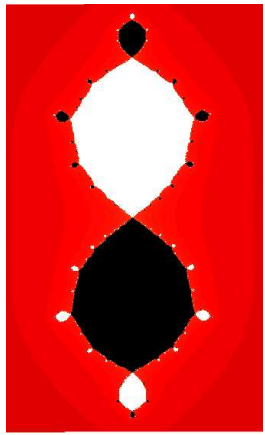}
  \end{center}
  \caption{The filled Julia set of the double basilica: $K(f_\dbas)$.}
  \label{Basilique}
\end{figure}

Hence,  up to affine conjugacy, we may choose $f_{\dbas}(z)=z(z^2+\frac32)$, which we  call the {\em double-basilica}. It has two super-attracting fixed points. The second polynomial lies in  the family  $f_a(z)=z^2(z+3a/2)$ for $a\in \C$. It always has a super-attracting fixed point at $0$. The other free critical point is $-a$.
The Julia set of $f_\dbas$ is connected and locally connected (see figure~\ref{Basilique}). For the polynomial $f_a$
we  concentrate on the connectedness locus denoted  by $\mathcal C$:$$\mathcal C:=\{a\in \C\mid K(f_a)\  \hbox{ is connected}\}\quad  \hbox{(see figure~\ref{Parameter-fa})}.$$

\begin{figure}[ht]
  \begin{center}
  \includegraphics[scale=0.4]{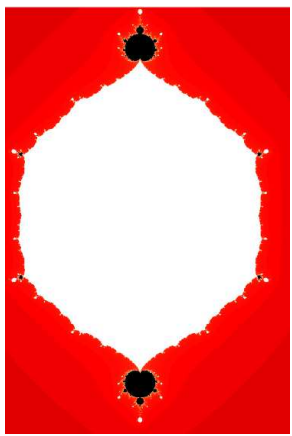}
  \end{center}
  \caption{The connectedness locus for $f_a$ is the complement of the red. The set $\HH_0$ is the central white region.  Small (black) Mandelbrot copies are attached to it.}
  \label{Parameter-fa}
\end{figure}

In the connectedness locus, the Julia set of $f_a$ is not always locally connected. We consider the set $$\HH:=\{a\in \C\mid f_a^n(-a)\to 0\}. $$
The maps in $\HH$ are  hyperbolic and have  locally connected Julia set. Let  $\HH_0$ be  the connected component containing the parameter $a=0$ of $\HH$ (the big white citrus component in
Figure~\ref{Parameter-fa}).

\begin{figure}[ht]
  \begin{center}
 \includegraphics[scale=0.62]{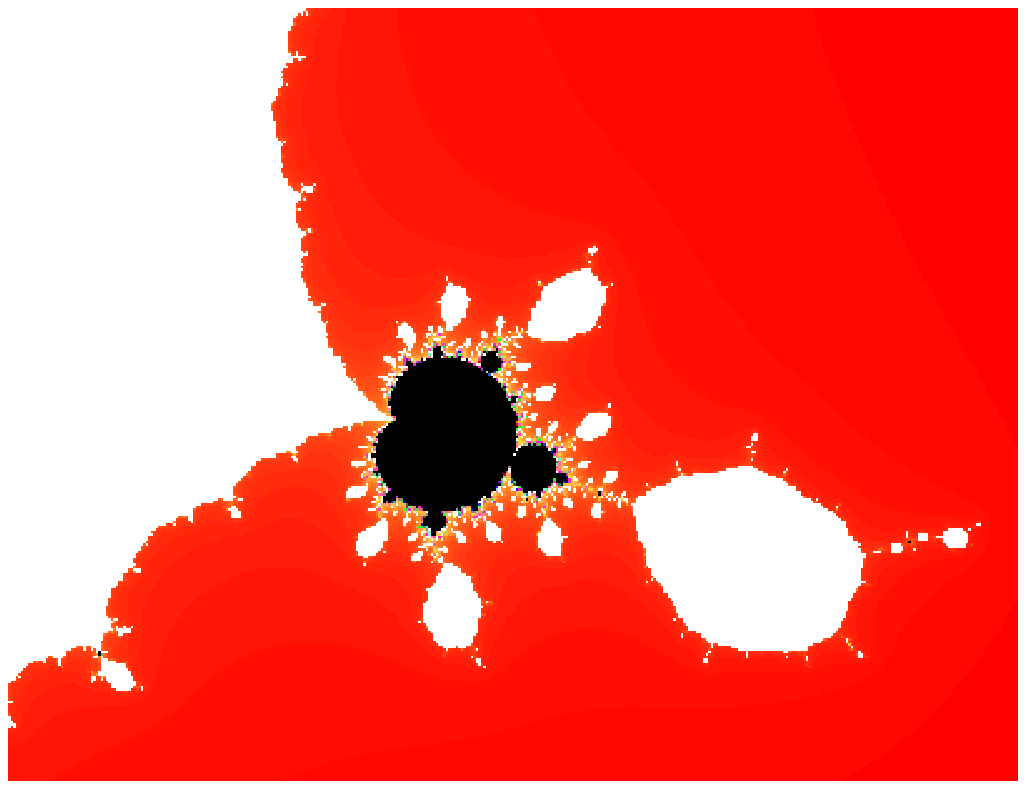}
\includegraphics[scale=0.3661]{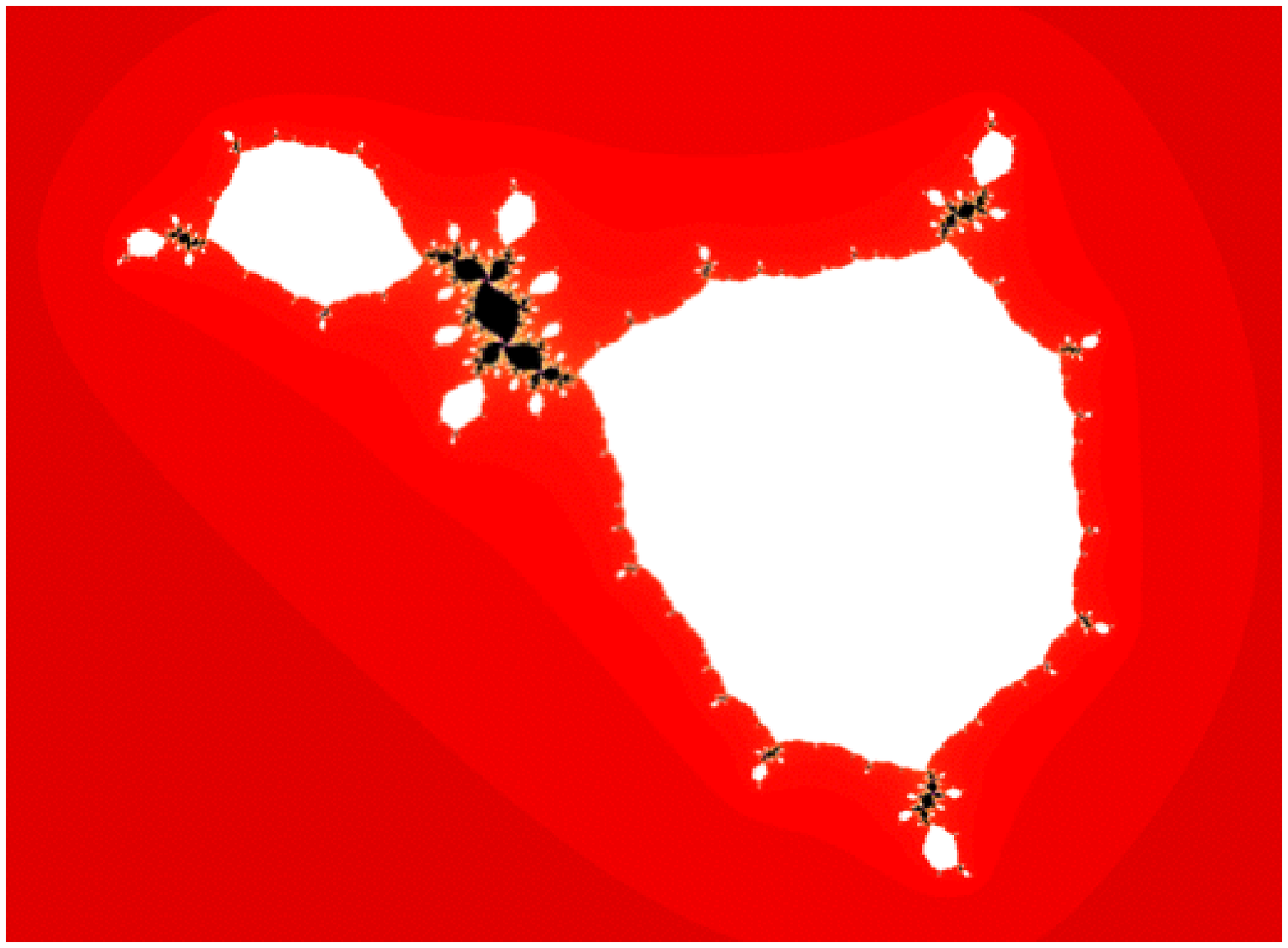}
  \end{center}
  \caption{Some copy of the Mandelbrot set attached to $\HH_0$ on the left and, on the right,  a Julia set  for the map $f_a$ in this copy. }
    \label{Parameterzoom}
\end{figure}

The boundary of $\HH_0$ is a Jordan curve (see~\cite{RoeschENS}) and can be parameterized nicely  by
a map $t\in \S^1\mapsto a(t)$ which contain some dynamical information. Each connected component of $\CC\setminus \overline{\HH_0}$  is attached to $\overline{\HH_0}$  by a parameter $a(t)$ for $t$ in some subset $T\subset \S^1$.
Moreover, each such parameter $a(t)$ with $t\in T$ is the cusp of a Mandelbrot copy, {\it i.e.} the image by some homeomorphism of the Mandelbrot set $\M:=\{c\in \C\mid J(z^2+c) \ \hbox {is connected}\}$, the cusp being the image of $c=1/4$. See~\cite{RoeschENS} for more details.

\begin{figure}[ht]
  \begin{center}
 \includegraphics[scale=0.5]{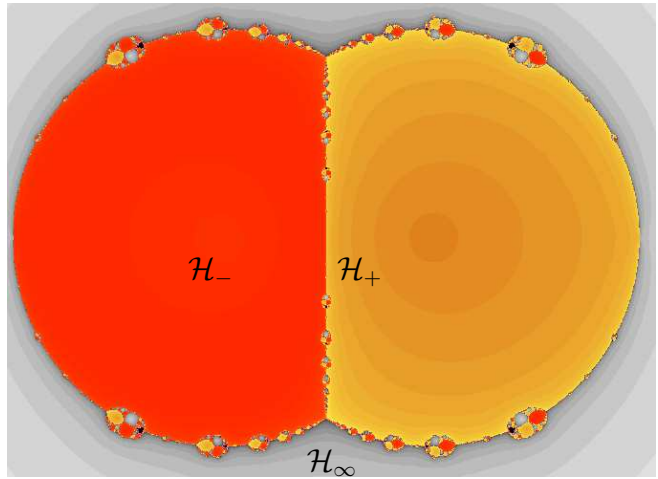}
  \end{center}
  \caption{Parameter space of cubic Newton maps. }
    \label{Parameternewt}
    \vskip -10 em \hskip -3 em $\HH_-$ \hskip 3 em  $\HH_+$  \vskip 5 em $\HH_\infty$
    \vskip 2 em
\end{figure}

Concerning the  cubic Newton maps, a description is given in~\cite{RoeschAnnals}.
Any cubic Newton map can be conjugated to a rational map of the form $N_\lambda(z)=z-\frac{P_\lambda(z)}{P_\lambda'(z)}$ where $P_\lambda(z)=(z+1/2-\lambda)(z+1/2+\lambda)(z-1)$. For $\lambda\notin\{-3/2,0,3/2\}$ it has three critical fixed points,  the roots of $P_\lambda$, denoted by $\hbox{Root}(\lambda)$,   the fourth critical point being the
point $x_0=0$. Denote by $\HH_+$, $\HH_-$, $\HH_\infty$  the connected component  of the set
$$\{\lambda\in \C\mid N^n_\lambda(0)\to \hbox{Root}(\lambda)\}$$ containing $1/2$, $-1/2$, $\infty$ respectively.
The boundary of this component $\HH_u$ is a Jordan curve (see~\cite{Roeschcras}). It  satisfies similar  properties as $\HH_0$\,;  there exist a parametrization $\lambda_u(t)$ (containing  dynamical information), such that for $t$ in some set  $T_u'$, there is a copy of the Mandelbrot set attached to $\partial \HH_u$ at the point $\lambda_u(t)$.

\begin{dfn} Let $RC$ denote the parameters  $a\in \CC$ belonging to the Mandelbrot copies attached to $\HH_0$ except the ones containing the points $a_\pm=\pm\frac{4i}3$ for the cubic family.  Let $RN$ denote the parameters in $\C$ belonging to the Mandelbrot copies attached
to $\HH_- $ for the Newton map.
\end{dfn}
The main result is the following\,:
\begin{thm}\label{resultat1}For any parameter $a\in RC$ the polynomials  $f_a$ and $f_{\dbas}$ are conformally mateable if $J(f_a)$ is locally connected. Moreover, up to conjugation by a M\"obius transformation there exists a  cubic Newton map realizing this mating. \end{thm}

With the same techniques, we also prove:
\begin{thm}\label{resultat2} For any parameter $a\in \partial \HH_0 $, $a \neq \pm \frac{4i}{3}$, the maps $f_a$
and $f_\dbas$ are conformally mateable, and their mating is a cubic Newton map.
\end{thm}

Let $NRC = \{ a \in \partial \HH_0 : \text{$a$ is not a cusp of a Mandelbrot copy} \}$.
So for $a \in NRC$, $f_a$ is not renormalizable around its free critical point $-a$.
Combining the results of Theorems \ref{resultat1} and \ref{resultat2} we have a map
$$\NN: RC \cup NRC   \to RN \cup \partial \Omega_-$$
which assigns to any cubic polynomial of $RC \cup NRC$ with locally connected Julia set,  a Newton map in $RN \cup \Omega_-$ such that this map is the mating of $f_a$ with $f_\dbas$.

\begin{rmk}The two Mandelbrot copies taken away in the set $RC$ (where $a = \pm 4 i/3$ are the cusps) have the property that for any parameter $a$ there,
the external rays $R(1/2)$ and $R(0)$ of angles $1/2$ and $0$ respectively land at the repelling (or parabolic) fixed point. This is also the case for the  double-basilica polynomial $f_{\dbas}(z) = z (z^2 + 3/2)$.  Clearly the two pairs of rays separate the sphere into two sets and $S^2/\sim_r$ cannot be homeomorphic to $S^2$.
Hence the mating between $f_a$ and $f_\dbas$ in this case does not exist by this   topological obstruction. \end{rmk}

The main part of the paper is devoted to the proof of Theorem \ref{resultat1}, while the proof of Theorem \ref{resultat2} is given in Section 7.

\section{Dynamical planes}
We first present the universal model given by B\"ottcher maps, then we study  the dynamical planes of the polynomials:
$f_\dbas(z)=z(z^2+3/2)$ and $f_a(z)=z^2(z+3a/2)$ and of the family of Newton maps $N_\l(z)=z-\frac{P_\l(z)}{P'_\l(z)}$ where
$P_\l(z)=(z+1/2-\l)(z-1/2+\l)(z-1)$.
Note  that  the cubic polynomial $f_\dbas$ belongs  to the family $f_a$  (up to conjugacy).
Indeed, for $a=\pm i\sqrt2$, the map $f_a$ has two fixed critical points, therefore it has to be conjugated to $f_\dbas$.
Nevertheless,  we give a separate study because the notations are different, and we hope that by concentrating on the case when
$a \neq \pm i \sqrt{2}$ the arguments will be more transparent.

\subsection{Preliminaries}

In this section we recall some basic facts about dynamics, but we refer to~\cite{M1} for more details.

Let $f$ be a rational map. Recall that the {\it Julia set} $J(f)$ of $f$ is the closure of repelling periodic orbits
(or, equivalently, the minimal compact totally invariant set containing at least $3$ points).
A {\it Fatou } component is by definition a connected component of the complement of the Julia set $J(f)$.

If  a  Fatou component  $U$ contains a critical point   which is  fixed by $f$ and contains no other critical points,
then $U$ is simply connected. Moreover, if the  degree of the critical point is $d\ge 2$, then  there exists a
Riemann map $\phi:U\to \D$ which satisfies  $\phi(f(z))=\phi(z)^d$.
If $d=2$  the map $\phi$ is unique. If $d>2$ there are multiple choices.

In this paper, a degree  $d>2$ will appear only for cubic polynomials at the point of $\infty$.
 In this   case
with the additional  assumption   that   $\phi$  is tangent to identity  at   $\infty$,  there is no choice on $\phi$.
In all cases, we call  this unique map  $\phi:U\to \D$  the  {\it B\"ottcher map} of $U$.

It allows to define  polar coordinates on $U$:\begin{itemize}
\item  a {\it ray}  of angle $t$  which is the set   $$\phi^{-1}( e^{2\pi i t }[0,1[) \hbox{ with }t \in \mathbf R/\mathbf Z$$
\item an  {\it equipotential} of level $r$ which is the set  $$\phi^{-1}( re^{2\pi i [0,1[ })\hbox{ with }r \in [0,1[.$$
\end{itemize}

Moreover, by a Theorem of Carath\'eodory, the conformal map $\phi^{-1}: \D \to U$ extends continuously to the boundary as soon as it is locally connected. As a consequence,  we may say
that a ray with angle $t$  {\it lands}  at  the point  $\gamma(t):=\phi^{-1}( e^{2\pi i t })$ for $t\in \R/\Z$.

In particular, for a cubic polynomial $P$, if the Julia set is locally connected, we have $\gamma(3t)=P(\gamma(t))$;
therefore  we introduce  now the triadic expansion for an angle.
\begin{dfn}
For any sequence $\{\epsilon_i\}_{i=0}^\infty$ of $\tilde{\Sigma}=\{0,1,2\}^\N$
we associate the angle  $$\theta =\displaystyle \sum_{i=1}^\infty
\frac{\epsilon_{i-1}}{3 ^i}\mod 1\hbox{ in }\R/\Z.$$
Let  $\sim$   be the  equivalence  relation on $\tilde{\Sigma}$ given by
 $$\overline 0\sim\overline 2, (\epsilon_0\cdots\epsilon_n1\overline 0)\sim (\epsilon_0\cdots\epsilon_n0\overline 2),  (\epsilon_0\cdots\epsilon_n2\overline 0)\sim (\epsilon_0\cdots\epsilon_n1\overline 2) $$   The  sequence $\epsilon_0\cdots\epsilon_n$ can be empty.
 Let  $\Sigma=\
 \tilde{\Sigma} /\sim$ be the quotient and denote by $[(\epsilon_i)] $ the projection in $\Sigma$ of a sequence of $(\epsilon_i)\in  \tilde{\Sigma}.$
 \end{dfn}The previous map  factors to  $\Sigma$.
 \begin{lem} The map  $\theta: \Sigma\to \R/\Z $ defined
  by     $$\theta(x)=\displaystyle \sum_{i=1}^\infty
\frac{\epsilon_{i-1}}{3 ^i}\mod 1 \hbox{ for  any } x=[\{\epsilon_i\}_{i=0}^\infty]\in \Sigma $$ is a bijection.
\end{lem}
\proof
 We build   the converse map as follows.
 Let  $t\in \R/\Z$  be any angle which is not triadic ({\it i.e.} not of the form $\displaystyle  \frac k{3^N}$).
We have a unique sequence  in $\tilde{\Sigma}$ defined  by the ``itinerary''  of $t$ with respect to the partition $\displaystyle \left\{0,\frac 13, \frac 23\right\}$
as follows:    $\epsilon=\{\epsilon_i\}_{i=0}^\infty$ where  $$3^it\in I_{\epsilon_i}= \left]\frac{\epsilon_i}3, \frac{\epsilon_i+1}3\right[\quad \forall i\ge 0.$$
 Note that for any non triadic $t\in \R/\Z$, the sequence  $\epsilon$ is not eventually $\overline 0$ nor $\overline 2$.
To reach a  contradiction assume that $\epsilon$  is  $\overline 0$ so that $3^it\in ]0,\frac 13[ $  and then  $0<t<\frac 1{3^{n+1}}$. When $n$ goes to infinity we obtain that $0<t\le 0$, a contradiction.

Now consider the triadic angles. First, we associate to the angle $0$ the two sequences $\overline{0}$ and $\overline 2$.
Then, for the angle $1/3$ we associate the sequences$1\overline{0}$ and $0\overline 2$, for the angle $2/3$ we associate the sequences  $2\overline{0}$ and $1\overline 2$.
Now any other triadic angle $\theta$ is an  iterated pre-image (under the multiplication by $3$) of $1/3$ or $2/3$. Let us  define $\epsilon_i$ by $ 3^i \theta\in I_{\epsilon_i}$ for $i<n$  where $ 3^n \theta\in \{1/3,2/3\}$. Then we concatenate the sequence $\epsilon_0\cdots \epsilon_{n-1}$ with the two sequences associates to  $ 3^{n} \theta$.

 Note that we found  two  equivalent sequences for triadic angles. So the pre-image of $t$ under $\theta$ is well defined as the equivalence class  of these  two sequences.
\endproof

\begin{dfn}\label{d:itangle}
Let  us  call the  {\it itinerary } of $\theta\in \R/\Z$ the unique $x\in \Sigma$ such that $\theta(x)=\theta$ and write $\epsilon(\theta):=x$.\end{dfn}

\subsection{The double basilica}

\begin{figure}[ht]
  \begin{center}
 \includegraphics[scale=0.52 ]{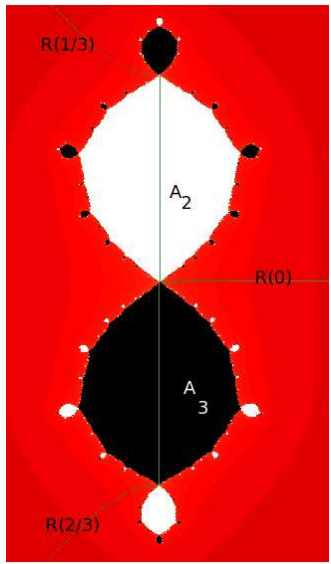}
  \end{center}
  \caption{The rays $R_\dbas(0)$ landing at the fixed point $p=0$,   $R_\dbas(1/3)$ landing at its preimage $p'$ and $R_\dbas(2/3)$ landing at the  other preimage $p''$, as well as internal rays. }
    \label{Basiliquerays}
\end{figure}

The map $f_\dbas(z)=z(z^2+\frac 32)$ has  two finite critical fixed points
$\pm \frac{i}{\sqrt2}$. Denote by $A_2$ and  $A_3$, the Fatou components containing  $i/\sqrt2$ and $-i/\sqrt2$ respectively (these  domains are also called the {\it  immediate basins of attraction} of the corresponding points). Denote by  $R^i_\dbas(t)$ the ray of angle $t$  and by  $E^i_\dbas (v)$ the equipotential of level $v$  in $A_i$. The {\it external} rays and equipotentials, corresponding to the unbounded Fatou component, are denoted by
 $R^\infty_\dbas(t)$ and   $E^\infty_\dbas(v)$.

The third finite fixed point called $p_\dbas$ is a common point of their boundaries: $$p_\dbas=\partial A_2\cap \partial A_3.$$
Indeed, the rays $R^\infty_\dbas(0)$, $R^\infty_\dbas(1/2)$,  $R^2_\dbas(0)$ and
  $R^3_\dbas(0)$  are fixed by $f_\dbas$, so  they land at the sole fixed point belonging to  the Julia set $J(f_\dbas)$, {\it i.e.} at $p_\dbas=0$.
\vskip 1em
A point of the Julia set  is called  {\em bi-accessible} if exactly two external rays land at it.
\begin{lem} \label{biaccess}
The bi-accessible points of $J(f_\dbas)$ are exactly  the iterated pre-images of $p_\dbas=0$.

Moreover, no other point of $J(f_\dbas)$ is the landing point of at least two   external rays.
\end{lem}

\begin{proof}
Let $x$ be a point with at least two external rays landing at it. Denote two of them by $R^\infty_\dbas(t)$ and $R^\infty_\dbas(t')$. We can assume that in the smallest interval of $\S^1\setminus \{t,t'\}$ there is no other angle $t''$  such that $R^\infty_\dbas(t'')$ lands at $x$. Denote by $\tau$ this smallest interval and  define the sector  $S$ to be the connected component of
$\C\setminus (\ol{R^\infty_\dbas(t)}\cup\ol{R^\infty_\dbas(t')})$ containing the rays $R^\infty_\dbas(u)$ for $u\in \tau$.
Note that if $S$ contains the rays $R^\infty_\dbas(0)$ it has to contain also $R^\infty_\dbas(1/2)$ and therefore
$\tau$ is not the smallest interval.  Hence, $S$ cannot contain neither $A_2$ nor $A_3$ where this rays land.
 Therefore the image of the sector $S$ is a sector between $3t$ and $3t'$.  If there is no   critical point  in the sector,   the  image of such a sector is still a sector. But since the size of the interval
 is multiplied by $3$ each time, some sector has to contain a critical point. Take the last image of the sector not containing the ray $R^\infty_\dbas(0)$.
 It is a sector containing $R^\infty_\dbas(1/3)$ (or $R^\infty_\dbas(2/3)$). Then the curve $\ol{R^\infty_\dbas(3^Nt)}\cup\ol{R^\infty_\dbas(3^Nt')}$ (bouding this sector) has to cross the curve $\ol{R^\infty_\dbas(0)}\cup\ol{R^\infty_\dbas(1/3)}\cup\ol{R^i_\dbas(0)}\cup\ol{R^i_\dbas(1/2)}$ where $\ol{R^\infty_\dbas(t)}$ is landing at the boundary of $A_i$ (in order to separate the rays). This is possible only at $p'$,  being the landing point of $R_\dbas(1/3)$. Therefore,
 $x$ is an iterated pre-image of $p_\dbas$.

 Now,  any external ray landing at $p=0$ has to be fixed by $f_\dbas$ (because all the rays landing at the same point have the same rotation number see~\cite{GM}). For this reason no other ray than  $R_\dbas(0)$ and  $R_\dbas(1/2)$
 land at $p=0$. By pull back, no points are accessible by more than two rays (otherwise two rays would have the same image and their landing point would be critical).
\end{proof}
\begin{crr}\label{c:biacc-angle}
The biaccessible points are exactly the landing points of external rays with triadic angles {\rm(}{\it i.e.} angle of the form $\displaystyle \frac k{3^m}${\rm)}.
\end{crr}

\begin{figure}[ht]
  \begin{center}
 \includegraphics[scale=0.52]{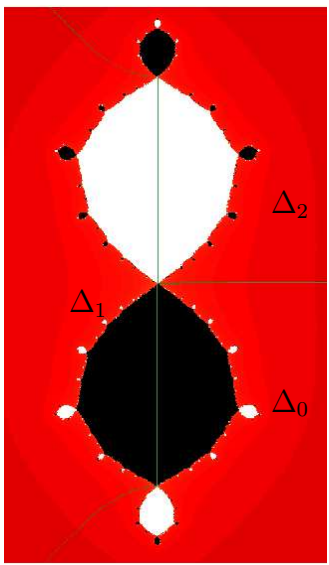}
  \end{center}
  \caption{The set $\Delta_1$ meets both immediate basins $A_2,A_3$. The set  $\Delta_2$ intersects $A_2$, but $\Delta_0$ does not. }
    \label{Partition}
    \vskip -16 em \hskip 8 em  $\Delta_2$
\vskip 2em \hskip -5 em $\Delta_1$
\vskip 2em \hskip 8 em  $\Delta_0$
\vskip 8 em
\end{figure}

\begin{dfn}\label{d:itidbas} Using the B\"ottcher map, we can define the itinerary class  $\epsilon_\dbas(z)$ of a point $z$  in the Julia set  $J(f_\dbas)$ as the set of  itineraries   $\epsilon(-t)\in \Sigma$ where $t\in\{\gamma_\dbas ^{-1}(z)\}$.
\end{dfn}

In particular
\begin{itemize}
 \item  $\epsilon_\dbas(p_\dbas)=\{[\overline 1],[\overline 2]=[\overline 0]\}$\,;
\item for $p'_\dbas=\gamma_\dbas(1/3)$   and $p''_\dbas=\gamma_\dbas(2/3)$ be   the pre-images of  $p_\dbas$. Then  we have    $$\epsilon_\dbas(p'_\dbas)=\{[2\overline 1], [1\overline 0]=[0\overline 2]\} \hbox{ and }\epsilon_\dbas(p''_\dbas)=\{[0\overline 1], [2\overline 0]=[1\overline 2]\}.$$
\end{itemize}

\begin{crr}Let $R_\dbas^\infty(t)$ be   a ray landing at   $z\in J(f_\dbas) $, then $\epsilon(-t) \in \epsilon_\dbas(z)$.
\end{crr}

Now we define a partition of $\C$ related to the triadic partition in the basin of $\infty$.

\begin{dfn} Let  $\Gamma_\dbas$ be the following graph\,:
$$\Gamma_\dbas=\overline{R_\dbas^\infty(0)}\cup \overline{R_\dbas^\infty(1/3)}\cup \overline{R_\dbas^\infty(2/3)}\cup\overline{ R^2_\infty(0) }\cup \overline{R^2_\infty(1/2)} \cup \overline{R^3_\infty(0) }\cup \overline{R^3_\infty(1/2)}. $$
It  cuts the sphere $\overline \C$ in three open connected components. Let  $\Delta^\dbas_1$ be the  component which  intersects both $A_2$ and $A_3$.
Denote by $\Delta^\dbas_0$ the component which  intersects only  $A_3$ and  $\Delta^\dbas_2$  the one that intersects only $A_2$.
\end{dfn}

This partition allows us  visualize the  itinerary classes of the points in  $J(f_\dbas)$. Some points of $J(f_\dbas)$ belong to the closure of   more than one  component $\Delta^\dbas_i$.

Note that any component of  $f^{-1}(\Delta^\dbas_{i})$ belongs to exactly one $ \Delta^\dbas_{j}$, since $\Gamma_\dbas$ is forward invariant.
Hence the intersection $f^{-1}(\Delta^\dbas_{i})\cap \Delta^\dbas_{j}$ just determines the component.

\begin{dfn} For any sequence $(\epsilon_i)_{i\in \N}\in \{0,1,2\}^\N$   we define $\Delta^\dbas_{\epsilon_0\ldots \epsilon_n}$  by the relation $$\Delta^\dbas_{\epsilon_0\ldots \epsilon_n} = f_{\dbas}^{-n}(  \Delta^\dbas_{\epsilon_n})\cap\Delta^{\dbas}_{\epsilon_{0} \ldots \epsilon_{n-1}}.$$
\end{dfn}

\begin{lem} \label{itin-dbas}
For any  point $z\in J(f_\dbas) $  we have $$[(\epsilon_i)_{i\in \mathbf N}]\in \epsilon_\dbas(z)
\iff z\in \bigcap_{n\in \mathbf N} \overline{\Delta^\dbas_{\epsilon_0\ldots \epsilon_n}}.$$
 \end{lem}
 \proof  Since $f_\dbas$ is proper then $$\overline{\Delta^\dbas_{\epsilon_0\ldots \epsilon_n} }= f_{\dbas}^{-n}(\overline{  \Delta^\dbas_{\epsilon_n}})\cap\overline{\Delta^{\dbas}_{\epsilon_{0} \ldots \epsilon_{n-1}}}.$$

  Now  from defintion we have the equivalence  that $[ (\epsilon_i)_{i\in \mathbf N}] \in \epsilon_\dbas(z) $ if and only if $  R_\dbas^\infty(t) \hbox{ lands at } z$ for $t=-\theta(\epsilon)$ so that $\epsilon(-t)=\{\epsilon_i\}_i$.  Then the ray $f_\dbas^j(R_\dbas^\infty(t))\subset\ol{ \Delta^\dbas_{\epsilon_j}}$ for every $j\ge 0$. We prove  by induction that $R_\dbas^\infty(t)\subset \overline{\Delta^{\dbas}_{\epsilon_{0} \ldots \epsilon_{n}}}$. This is clear for $n=0$. Assume that it is true for  some $n$, that is $R_\dbas^\infty(t)\subset \overline{\Delta^{\dbas}_{\epsilon_{0} \ldots \epsilon_{n}}}$. Since  $f_\dbas^{n+1}(R_\dbas^\infty(t))\subset\ol{ \Delta^\dbas_{\epsilon_{n+1}}}$,   then $R_\dbas^\infty(t))\subset f_{\dbas}^{-(n+1)}(\overline{  \Delta^\dbas_{\epsilon_{n+1}}}) $. Therefore    $R_\dbas^\infty(t)\subset \overline{\Delta^{\dbas}_{\epsilon_{0} \ldots \epsilon_{n+1}}}$
  since
  $\overline{\Delta^\dbas_{\epsilon_0\ldots \epsilon_{n+1}} }= f_{\dbas}^{-(n+1)}(\overline{  \Delta^\dbas_{\epsilon_{n+1}}})\cap\overline{\Delta^{\dbas}_{\epsilon_{0} \ldots \epsilon_{n}}}.$
  Finally we get that
  $$
   z\in \bigcap_{n\in \mathbf N} \overline{\Delta^\dbas_{\epsilon_0\ldots \epsilon_n}}.$$  \endproof

 We now define {\it puzzle pieces}.

 \begin{dfn} Let   $R\in ]0,1[$ and define $V_0$ to  be  the connected component  containing $J(f_\dbas)$ of $$\C\setminus (E^\infty_\dbas(R)\cup E^2_\dbas(R)
\cup E^3_\dbas(R))$$ (the  complement of the equipotentials of level $R$) and denote by   $V_n$ the preimage $f_\dbas^{-n} (V_0)$.
A {\it puzzle piece of level $n$ }  with itinerary $\epsilon_\dbas=\{\epsilon_i\}_{i=0}^\infty$ is  the set
 $$P^\dbas_{\epsilon_0 \dots \epsilon_n}= V_n\cap \Delta^\dbas_{\epsilon_0\cdots \epsilon_n}.$$
\end{dfn}

\begin{rmk}
Any point
$z\in J(f_\dbas)\setminus \cup f_\dbas^{-n}(p_\dbas)$ belongs to a unique  nested sequence  $(P^\dbas_{\epsilon_0 \dots \epsilon_n})$ whereas  points in $ \bigcup_{n\ge } f_\dbas^{-n}(p_\dbas)$  belong to a finite number of nested sequences  $(\overline {P^\dbas_{\epsilon_0 \dots \epsilon_n}})$
\end{rmk}

\begin{lem}\label{l:itidbas} For any sequence   $\{\epsilon_i\}_{i=0}^\infty$  in $\{0,1,2\}^\N$, the intersection
 $$\bigcap_{n\in \N}\overline{P^\dbas_{\epsilon_0 \dots \epsilon_n}}$$
 reduces to one point.  Moreover, for this point $z$  we have  $[\{\epsilon_i\}_{i=0}^\infty]\in \epsilon_\dbas(z)$.
 \end{lem}
\proof
Since $f_\dbas$ is a hyperbolic polynomial, there is an expanding hyperbolic metric on a neighbourhood of the Julia set. This implies that
puzzle pieces shrink to points exponentially fast. The last statement follows from Lemma \ref{itin-dbas}. 
\endproof

\subsection{The cubic family }

We now consider the following family of cubic polynomials: $$f_a(z)= z^2(z+3a/2) \text{ with } a\in \C.$$
There are two finite critical points, $0$ and $-a$. Since $0$ is fixed, denote by $A_1$ the Fatou component containing $0$ (the immediate basin of attraction of $0$). Note that $A_1$ depends on $a$. Recall that when $-a\notin A_1$ then there is a  (unique) B\"ottcher coordinate $\phi_a: A_1 \to \D$. Let $R_a^1(t)$ be the rays in $A_1$ of angle $t$ and $E_a^1(v)$ the equipotential of level $v$ in $A_1$. Let $A_\infty$ be the unbounded Fatou component. If $-a \notin A_\infty$ we similarly can use the B\"ottcher  map $\phi_a^\infty:A_\infty \rightarrow \D$ to define external rays and equipotentials. In this case let $R_a^\infty (t)$ be the external ray of angle $t$ and $E_a^\infty(v)$ be the equipotential of level $v$ in $A_\infty$.
We assume that the Julia set $J(f_a)$ is locally connected so that the inverse of the B\"ottcher maps   $\phi_a^{-1}:\D\to  A_1 $ and  $\phi^\infty _a:   \D\to A_\infty$ extends continusously  to the circle and define maps  $\delta_a:\mathbf S^1\to \partial A_1$ and $\gamma_a:\mathbf S^1\to \partial A_\infty$.

\begin{lem}\label{real} Suppose that $-a \notin A_1 \cup A_\infty$. For $a>0$  the rays   $R_a^{\infty}(0)$ and  $R_a^{0}(0)$ land at the same point $\delta_a(0)$.
\end{lem}
\begin{proof}If we take a real $a>0$, the map is real an can be easily studied.  There are three fixed points; $z=0$, $z=q_a > 0$ and $z=q_a' < 0$. The intervals $[q_a,+\infty)$ and $(-\infty,q_a']$ are fixed by $f_a$. Since $\phi^\infty_a$ is tangent to the identity at $\infty$ it follows that $R_a^{\infty}(0)=[q_a,+\infty)$, $R_a^{\infty}(1/2)=(-\infty,q_a']$. Moreover, the map $f_a(x)-x$ changes signs between the points $q_a',0,q_a'$. This implies that $f_a(x) < x$ for $x \in (0, q_a)$. Since $a > 0$ and $x > 0$ $f_a(x) > 0$ and hence every point in $(0,q_a)$ converges to zero under iteration. Hence $(0,q_a) \subset A_1$. Since $-a \notin A_1$, $q_a$ is the only fixed point on the boundary of $A_1$ and therefore $R_a^\infty(0)$ lands at $q_a$.
\end{proof}

\begin{crr}\label{correal} For $a\notin A_1 \cup A_\infty$ with $\Re e(a)>0$, the rays  $R_a^{\infty}(0)$ and $R_a^{0}(0)$ land at the same point $\delta_a(0)$.
\end{crr}
\begin{proof} We want to prove that in $(A_1 \cup A_\infty)^c$, $\Re e(a) > 0$, the closure of the dynamical rays $R_a^{\infty}(0)$ and $R_a^{0}(0)$ are stable.  Rays are stable as long  as their closure do no meet neither an iterated pre-image of  the critical point nor of a parabolic point (see~\cite{DH1}).The rays $R_a^{\infty}(0)$ and $R_a^{0}(0)$ both land at a fixed points, which obviously are not critical. Assume now that such a fixed point is parabolic. Then it must have multiplier equal to $1$. The only parameter $a$ such that $f_a$ has a double fixed point is equal to $\pm 4i/3$. But since $\Re e(a) > 0$ this cannot happen.

By the above lemma the rays $R_a^{\infty}(0)$ and $R_a^{0}(0)$ land at the common point $\delta_a(0)$ for $a > 0$, $a\notin A_1 \cup A_\infty$. Hence they have to do that throughout $(A_1 \cup A_\infty)^c$ for $\Re e(a) > 0$.
\end{proof}

Define the {\em filled Julia set} by $K(f_a) = \overline{\C} \setminus A_\infty$. The following proposition comes from ~\cite{RoeschENS}.

\begin{prop} Without assuming  the local connectivity of $J(f_a)$ we have that
if $-a\notin  A_1\cup A_\infty$ then, $\partial A_1$ is a Jordan curve. Therefore we can use the  parameterization by the extension $\delta_a(t)=\phi_a^{-1}(e^{2i\pi t })$.

Moreover, let $L$ be a non empty connected component of $K(f_a) \setminus \overline{A_1}$. Then $\overline{L} \cap \overline{A_1}$ is only one point.
\end{prop}
\begin{dfn}
For $t \in \R/\Z$, if $\delta_a(t)$ belongs to the closure of a connected component of $K(f_a) \setminus \overline A_1$, we call this closed connected component $L_t^a$, otherwise we define $L_t^a$ to be $\delta_a(t)$.
In other words,
$$K(f_a)=\overline{ A_1}\cup\bigsqcup_{t\in \R/\Z} L_t^a.$$
We call $t_0$ the {\em critical angle} if $-a \in L_{t_0}^a$.
\end{dfn}

\begin{prop}
If $-a \notin A_1 \cup A_\infty$ then the set $L_t^a$ is not empty if and only if $2^n t = t_0 (\mod 1)$ , for some $n \geq 0$.
\end{prop}

\begin{lem}
If $t_0 \in (0,1)$ the rays $R_a^{\infty}(0)$ and $R_a^{\infty}(1/2)$ cannot land at the same point.
\end{lem}
\begin{proof}
Suppose the contrary, i.e. $R_a^{\infty}(0)$ and $R_a^{\infty}(1/2)$ land at the same point $\delta_a(0)$. These two rays cut out a dynamical wake $W$ defined by the connected component of the complement not containing the basin $A_1$. This sector has to contain a critical point so $-a$. Indeed, the external rays of angle $2/3$ and $5/6$   (or $1/6$ and $1/3$)  belong to $W$ and land at a pre-image of $\delta_a(0)$ (the other one being on the boundary of $A_1$ so not in $W$). The region  between these two rays and $R_a^{\infty}(0)\cup \ol{R_a^{\infty}(1/2)}$ is then a disk map onto its image with degree $2$. The conclusion follows.
\end{proof}

Recall that a map $f$ is $k$-{\it renormalizable}  around a critical point $c$ if there are two topological disks $U, V$ containing $c$, with $\ol U\subset V$ and $f^k:U\to V$ is a proper holomorphic map satisfying $f^{kn}(c)\in V$ for all $n\ge 0$. The renormalization is the map $f^k$ and its filled Julia set is $K:=\bigcap_{n\in \N} f^{-kn}(U)$ (it is connected). By~\cite{DH2}, there exists a unique $c\in \C$ and a quasi-conformal homeomorphism $\sigma_f$ defined on a neighborhood of $K$ such that $\sigma_f(K)=K(P_c)$ where $P_c(z)=z^2+c$ and $\sigma_f\circ f^k=P_c\circ \sigma_f$.

\vskip 1em
In order to define the itineraries as before we consider the following assumptions for $f_a$.
\begin{assu}\label{assumcubic} \phantom{.} \quad
\begin{itemize}
\item The Julia set $J(f_a)$ is locally connected
\item $t_0 \neq 0$
\item $-a \notin A_1 \cup A_\infty$,
\item The critical angle $t_0$ is $k$-periodic under multiplication by $2$
\item $f_a$ is $k$-renormalisable around $-a$ and that its filled Julia set $K_a$ intersects $\overline{A_1}$ only at $\delta_a(t_0)$.
\end{itemize}
\end{assu}

\begin{dfn}\label{d:itidfa} Using the B\"ottcher map, we can define the itinerary class  $\epsilon_a(z)$ of a point $z$  in the Julia set  $J(f_a)$ as the set of  itineraries   $\epsilon(t)\in \Sigma$ where $t\in\{\gamma_a ^{-1}(z)\}$.
\end{dfn}

In particular
\begin{itemize}
 \item  For $p_a=\delta_a(0)=\gamma_a(0)$ we have $\epsilon_a(p_a)=\{[\overline 2]=[\overline 0]\}$\,;
\item for $p'_a=\gamma_a(1/3)$   and $p''_a=\gamma_a(2/3)$    the pre-images of  $p_a$. Then  we have    $$\epsilon_\dbas(p'_a)=\{[1\overline 0]=[0\overline 2]\} \hbox{ and }\epsilon_\dbas(p''_a)=\{ [2\overline 0]=[1\overline 2]\}.$$
\end{itemize}

\begin{crr}\label{c:angleitin}Let $R_a^\infty(t)$ be   a ray landing at   $z\in J(f_a) $, then $\epsilon(t) \in \epsilon_a(z)$. Conversely,  if $\epsilon \in \epsilon_a(z)$ then    $R_a^\infty(\theta)$  lands at   $z\in J(f_a) $ where $\theta=\theta(\epsilon)$.
\end{crr}
\proof  If  $R_a^\infty(t)$ is  landing at   $z\in J(f_a) $, then $\gamma_a(t)=z$.   Then  $\epsilon(t)\in \epsilon_a(z)$ because  $t\in\{\gamma_a ^{-1}(z)\}$. Now, if $\epsilon$ belongs to $\epsilon_a(z)$ then  by definition $\epsilon=\epsilon(t)$ for some $t=\theta(\epsilon)$ (see Defintion~\ref{d:itangle}).
Then by definition we have $t\in\{\gamma_a ^{-1}(z)\}$. So that $z$ is the landing point of $R^\infty_a(t)$.\endproof

By definition we get a characterization of multiply accessible points:
\begin{crr}\label{c:biacc}
$z\in J(f_a)$ is  multilply accessible   if and only if  its itinerary class is not reduced to one point.
\end{crr}
\proof  If $\epsilon,\epsilon'$ define $t=\theta(\epsilon)$ and $t'=\theta(\epsilon')$. Then  $z$ is the landing point of $R^\infty_a(t)$ and $R^\infty_a(t')$ if  and only if $\epsilon(t), \epsilon(t')\in\epsilon_a(z)$ by Corrollary~\ref{c:angleitin}. Moreover,
by the formula of $\theta(x)$ the angles $t$ and $t'$ are different if and only if $\epsilon\neq\epsilon'$.
\endproof

Now we define a partition of $\C$ related to the triadic partition in the basin of $\infty$.

\begin{dfn} Let  $A_1'$ denote the pre-image of $A_1$ and $R_a'(t)=f_a^{-1}(R_a^0(t))\cap A_1'$.
Let $\Gamma_a$ be the following set\,:
$$\Gamma_a=\overline{R^0_a(0)} \cup \overline{R_a^\infty(0)} \cup \overline{R_a^\infty(1/3)} \cup \overline{R^0_a(1/2)} \cup \overline{R_a^\infty(2/3)} \cup \overline{R'_a(0)} \cup \Gamma'_a,$$ $$
\hbox{ where }\displaystyle  \Gamma'_a=\bigcup_{n\ge 0} f_a^n \left( R'_a(2t_0)\cup K_a\right).$$
\end{dfn}

\begin{figure}[h]
  \begin{center}
  \includegraphics[scale=0.52]{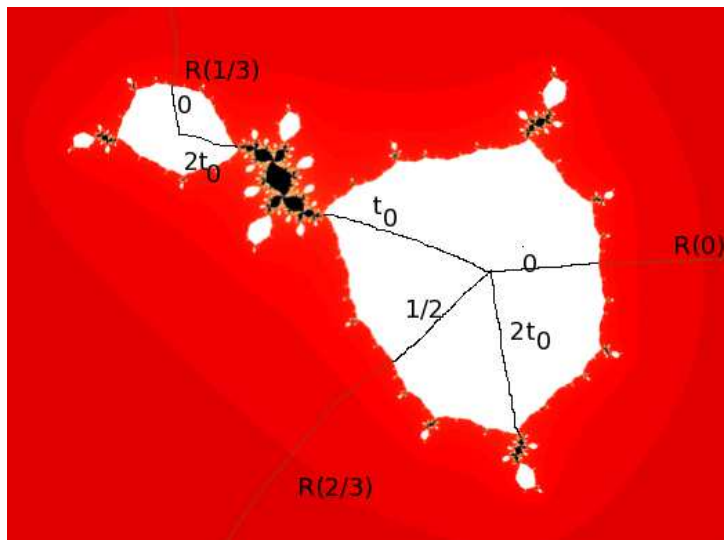}
 \end{center}
  \caption{ The graph $\Gamma_a$}
    \label{graph}
\end{figure}

\begin{lem} \label{component}
The set $\Gamma_a$ is connected. Moreover, the set $\C \setminus \Gamma_a$ is a union three open connected and simply connected components.
\end{lem}
\proof
The fact that $\Gamma_a$ is  compact and connected is obvious from the construction.
Its complement consists of open connected components $D_1,D_2,\ldots$. But $\Gamma_a \cup (\cup_{n \neq j} D_n )$ is connected for each $j$ and also equal to the complement of $D_j$, hence each $D_j$ is simply connected.

Note also that, on the Riemann sphere, the sets $E_1 = \{ \infty \} \cup R_a^{\infty}(1/3) \cup R_a'(0) \cup R_a'(2t_0) \cup K_a \cup R_a^0(t_0)$, $E_2 = \{ \infty \} \cup R_a^0(0) \cup R_a^{\infty}(0)$ and $E_3 = \{ \infty \} \cup R_a^0(1/2) \cup R_a^{\infty}(2/3)$ are connected, connect $\infty$ with $0$ and meet only at $0$ and $\infty$. Hence their complement consists of three components. \endproof

Let us now label these components.
\begin{dfn}
Let $\Delta^a_2$ be the component from Lemma \ref{component} containing $f_a(K_a)$ in its closure. Denote  the two others by  $\Delta^a_0$,   $\Delta^a_1$, where $\Delta^a_1$ contains $R_a^\infty(1/2)$.
\end{dfn}

\begin{figure}[ht]
  \begin{center}
 \includegraphics[scale=0.5]{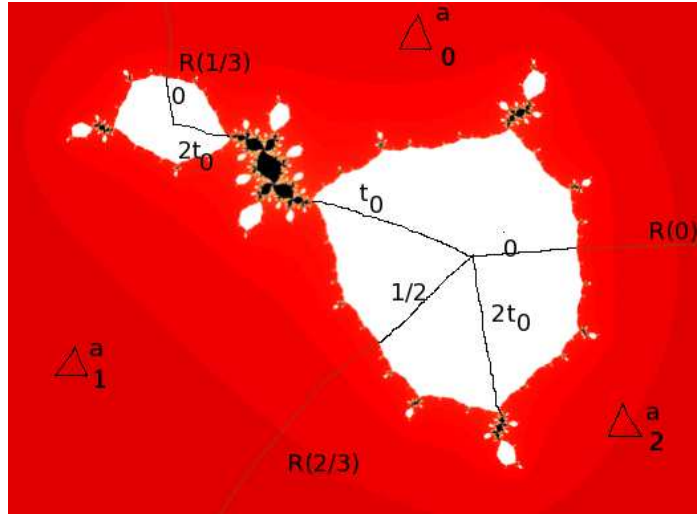}
  \end{center}
  \caption{Partition by $\Gamma_a$}
    \label{Partition2}
\end{figure}

This partition allows us  visualize the  itineraries of the points in  $J(f_a)$. Some points of $J(f_a)$ belong to the closure of   more than one  component $\Delta^a_i$.

We now give another formulation of this fact.
Note that any component of  $f^{-1}(\Delta^a_{i})$ belongs to exactly one $ \Delta^a_{j}$, since $\Gamma_a$ is forward invariant.
Hence the intersection $f^{-1}(\Delta^a_{i})\cap \Delta^a_{j}$ just determines the component.

\begin{dfn} For any sequence $(\epsilon_i)_{i\in \N}\in \{0,1,2\}^\N$   we define $\Delta^a_{\epsilon_0\ldots \epsilon_n}$  by the relation $$\Delta^a_{\epsilon_0\ldots \epsilon_n} = f_{a}^{-n}(  \Delta^a_{\epsilon_n})\cap\Delta^{a}_{\epsilon_{0} \ldots \epsilon_{n-1}}.$$
\end{dfn}

\begin{lem}\label{l:itfa} For any  point $z\in J(f_a) $  we have $$[(\epsilon_i)_{i\in \mathbf N}]\in \epsilon_a(z)
\iff z\in \bigcap_{n\in \mathbf N} \overline{\Delta^a_{\epsilon_0\ldots \epsilon_n}}.$$
 \end{lem}

 \proof  It follows from the fact that since $f_a$ is proper then $$\overline{\Delta^a_{\epsilon_0\ldots \epsilon_n} }= f_{a}^{-n}(\overline{  \Delta^a_{\epsilon_n}})\cap\overline{\Delta^{a}_{\epsilon_{0} \ldots \epsilon_{n-1}}}.$$

  Now  from defintion we have the equivalence  that $[ (\epsilon_i)_{i\in \mathbf N}] \in \epsilon_a(z) $ if and only if $  R_a^\infty(t) \hbox{ lands at } z$ for $t=\theta(\epsilon)$ so that $\epsilon(t)=\{\epsilon_i\}_i$.  Then the ray $f_a^j(R_a^\infty(t))\subset\ol{ \Delta^a_{\epsilon_j}}$ for every $j\ge 0$. We prove  by induction that $R_a^\infty(t)\subset \overline{\Delta^{a}_{\epsilon_{0} \ldots \epsilon_{n}}}$. This is clear for $n=0$. Assume that it is true for  some $n$, that is $R_a^\infty(t)\subset \overline{\Delta^{a}_{\epsilon_{0} \ldots \epsilon_{n}}}$. Since  $f_a^{n+1}(R_a^\infty(t))\subset\ol{ \Delta^a_{\epsilon_{n+1}}}$,   then $R_a^\infty(t))\subset f_{a}^{-(n+1)}(\overline{  \Delta^a_{\epsilon_{n+1}}}) $. Therefore    $R_a^\infty(t)\subset \overline{\Delta^{a}_{\epsilon_{0} \ldots \epsilon_{n+1}}}$
  since
  $\overline{\Delta^a_{\epsilon_0\ldots \epsilon_{n+1}} }= f_{a}^{-(n+1)}(\overline{  \Delta^a_{\epsilon_{n+1}}})\cap\overline{\Delta^{a}_{\epsilon_{0} \ldots \epsilon_{n}}}.$

  Finally we get that
  $$
   z\in \bigcap_{n\in \mathbf N} \overline{\Delta^a_{\epsilon_0\ldots \epsilon_n}}.$$  \endproof

We now define puzzle pieces for the map $f_a$.
\begin{dfn} Let   $R\in ]0,1[$ and define $W_0$ to  be  the connected component  containing $J(f_a)$ of $$\C\setminus (E^\infty_a(R)\cup E^1_a(R))$$ (the  complement of the equipotentials of level $R$) and denote by   $W_n$ the preimage $f_a^{-n} (W_0)$.
A {\it puzzle piece of level $n$ }  with itinerary $\epsilon_a=\{\epsilon_i\}_{i=0}^\infty$ is  the set
 $$P^a_{\epsilon_0 \dots \epsilon_n}= W_n\cap \Delta^a_{\epsilon_0\cdots \epsilon_n}.$$
\end{dfn}

Similar to the double basilica we have the following.
\begin{lem}\label{l:itifa} For any sequence   $\{\epsilon_i\}_{i=0}^\infty$  in $\{0,1,2\}^\N$, the intersection
 $$\bigcap_{n\in \N}\overline{P^a_{\epsilon_0 \dots \epsilon_n}}$$
 reduces to one point.  Moreover, for this point $z$  we have  $[\{\epsilon_i\}_{i=0}^\infty]\in \epsilon_a(z)$.
 \end{lem}
\begin{proof}
 We want to consider another puzzle where we know that puzzle pieces have the desired property and then compare it to the original puzzle for $f_a$. In~\cite{RoeschENS} and in ~\cite{DR}  it is proved (using a special puzzle) that the Julia set is locally connected as soon as the small Julia set is locally connected. There we use another graph which is the following:  $$\tilde \Gamma_0 = \bigcup_{i\ge 0}R_a^0(2^it)\cup \ol{R_a^\infty(3^i t')}\quad \tilde \Gamma_n=f_a^{-n}(\tilde \Gamma_0).$$
 The angle  $t$ is  any periodic angle  (periodic by multiplication by $2$)   with period sufficiently large and $t'$ is defined such that   $R_a^\infty(t')$ lands at the same point as $R_a^0(t)$.
    The puzzle pieces considered there are connected  components of $W_n\setminus  \tilde \Gamma_n$. Let us call the puzzle for $f_a$ defined by the graph $\Gamma_a$ in this paper {\em the original puzzle} and the puzzle from the graph $\tilde \Gamma_0$ the {\em  new Puzzle}.

Denote by $ S_n(z)$ the puzzle piece of depth $n$ for this graph which contains the point $z$. From~\cite{RoeschENS} and in ~\cite{DR}, we know  that the intersection    $\cap _{n\ge 0} \ol{S_n(z)}$ is either the point $z$ or a preimage of the small Julia set $K_a$.

Consider the {\em refinement} of the original puzzle with the new puzzle; i.e. the intersection of all puzzle pieces from both the
original puzzle and the new puzzle. In other words, refined puzzle pieces are the connected
components of the complement of the union of the graphs $\tilde \Gamma_0$ and $\Gamma_a$. Clearly, nests of refined puzzle pieces
shrink to points since the pieces $S_n(z)$ do. Note also that the original puzzle pieces never contain critical points.
Hence $f_a^n$ is univalent on any original puzzle piece $P$ depth $n$.
It is easy to see that there is some $K < \infty$ such that any original puzzle piece of depth $0$ consists of at most
$K$ refined puzzle pieces of depth $0$. Since $f_a^n$ is univalent on original puzzle pieces of depth $n$ it follows that any original
puzzle piece of depth $n$ consists of at most $K$ refined pieces of depth $n$, for all $n \geq 0$.
Hence, the nest of the original puzzle pieces also shrink to points or (subsets of) iterated preimages of the small Julia set.

In the first case, the intersection $ \cap_{n\ge 0} \ol{P_{\epsilon_0\cdots \epsilon_n}}$ is clearly reduced to exactly one point.

In the second case  $\bigcap_{n\in \N} \ol{P^a_{\epsilon_0\cdots \epsilon_n}}\subset K_a$. We consider the straightening map
$\sigma$ defined in a neighborhood of $K_a$ to a neighborhood of $K(P_c)$ conjugating $f_a$ to $P_c$ for some $c$.
Then the image  $\sigma( \ol{P^a_{\epsilon_0\cdots \epsilon_n}})$ will be included in a puzzle piece of level $n$ of the nest defined as follows.
Let $\phi_{Pc}$ be  the B\"ottcher coordinate at infinity  of  the quadratic polynomial $P_c$.
It is then clear that $$\phi_{Pc}(\sigma( \ol{P^a_{\epsilon_0\cdots \epsilon_n}}))\subset G^a_{\epsilon_0\cdots \epsilon_n}$$ where   $$G^a_{\epsilon_0\cdots \epsilon_n}=
\{z\in \C\mid \arg(z)\in [t_n,t'_n], \ 1\le \vert z\vert \le R^{1/2^n}\} $$ for some dyadic $t_n,t'_n$ with $\vert t_n-t'_n\vert\le 1/2^n$ and  for some $R>1$.
 Since we assume that the Julia set of $P_c$ is locally connected, we can conclude 
that these nests of puzzle pieces  for $P_c$ shrinks to a
point and therefore since $\sigma$ is an homeomorphism, the corresponding nests of (original) puzzle pieces shrinks to a point.

The second  part  of the Lemma comes from Lemma~\ref{l:itfa} since  $$\overline{P^a_{\epsilon_0 \dots \epsilon_n}}= \overline {W_n}\cap \overline {\Delta^a_{\epsilon_0\cdots \epsilon_n}}.$$
\end{proof}

 Note that external rays with  triadic angles  do not land at multiple accessible points  as for $f_\dbas$.  More precisely, we describe now  the points with several itineraries which are exactly the points with multiple  external accesses.

\begin{lem}\label{l:multiaccess}
If several external rays land at the same point of $J(f_a)$,  then some iterate of this point belongs to $K_a$.
\end{lem}
\begin{proof}
Suppose $x$ is the  landing point of  external rays of angles $\theta_1$ and $\theta_2$.  Then some preimage   $\Gamma_n= f_a^{-n}(\Gamma_a)$,  of the graph for $f_a$ will cross this curve $C=R_a^\infty(\theta_1)\cup \ol{R_a^\infty(\theta_2)}$. Indeed, either  $\Gamma_a$ crosses $C$ or $\Gamma_a$ belongs to one connected component  $U$ of $\C\setminus C$. In the last case, when pulling back  one ray of the form $p/3^n$ will be in $\C\setminus U$. Then,
since for $m \leq n$  we have $\Gamma_m \subset \Gamma_n$, the curve $C$ crosses $\Gamma_n$.

 Hence we must have $x \in \Gamma_n$. Hence $f_a^n(x) \in \Gamma_a$. So both rays $3^n\theta_1$ and $3^n\theta_2$ land at $f_a^n(x)$ in $\Gamma_a$. The only possibility for this is when $f_a^n(x) \in K_a$ or an iterate of $K_a$.
\end{proof}

\subsection{Cubic Newton maps}

Any cubic Newton map can be conjugated to a rational map of the form $N_\lambda(z)=z-\frac{P_\lambda(z)}{P_\lambda'(z)}$ where $P_\lambda(z)=(z+1/2-\lambda)(z+1/2+\lambda)(z-1)$. For $\lambda\notin\{-3/2,0,3/2\}$ it has three critical fixed points,  the roots of $P_\lambda$  and a the fourth critical point being the
point $x_0=0$. We will describe the dynamics under the following assumptions corresponding to parameters described in the next section.
\begin{figure}[!h]
  \begin{center}
 \includegraphics[scale=0.4]{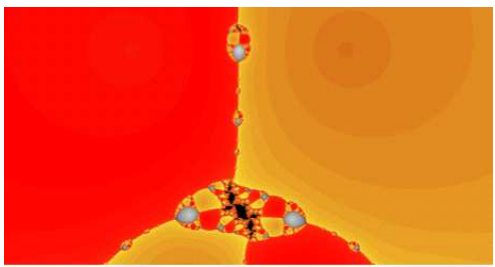}
  \end{center}
  \caption{The Julia set of a renormalizable  Newton map}
    \label{lapinattache}
\end{figure}

If the critical point $x_0=0$ is in the immediate basin of attraction of one of the roots of $P_\l$, then $N_\l$ is quasi conformally conjugated to a  cubic polynomial.

We suppose  that  $x_0=0$ is not in the immediate basin of attraction of one of the roots of $P_\l$.
Then, in each of the fixed immediate basin there is a  B\"ottcher coordinate which  is uniquely defined and  this defines  internal rays.
The following lemmas can be found in~\cite{TL} and also in~\cite{ RoeschAnnals}.
\begin{lem}The three rays of angle $0$ land at $\infty$. \label{threerays}
\end{lem}
 \begin{lem} \label{tworays}
 If the critical point $x_0=0$ is  not on the boundary of the three immediate fixed basins then only two of the three rays of angle $1/2$ land at  the same point (which is a preimage of $\infty$).
 \end{lem}

 \begin{dfn} Assume that  the critical point $x_0=0$ is not on the closure of the immediate basins of the roots. We call the immediate basins $B_1,B_2,B_3$ such that  the rays $R_1(0), R_2(0), R_3(0)$ (in $B_1, B_2, B_3$ respectively) meets in this cyclic order at $\infty$ and $R_1(1/2)$ lands at the same point as $R_2(1/2)$.

 Denote by $W_i$ the pre-image of $B_i$ distinct from $B_i$ and by $R'_i(t_0)=N_\lambda^{-1}(R_i(t))\cap W_i$. Then $R_1(0), R_2(0)$ land at the same point as $R_3(1/2)$ and  $R'_1(0), R'_2(0)$ land at the same point as $R_3(0)$.
 \end{dfn}

The following property is proved    in the article~\cite{RoeschAnnals}.
\begin{thm} If the critical point $0$ is not in the immediate basin of the roots, then the boundary of each connected component of the fixed basins of attraction is a Jordan curve.
\end{thm}

Therefore  we can extend in each of the fixed immediate basin the B\"ottcher coordinate (which  is uniquely defined).
We get a parametrization of the boundary by landing point of internal rays.

 In order to define itineraries as before we need to define puzzle pieces and therfore some graph.
 In order to do this we will consider the following assumptions:

\begin{assu}\phantom{.} \quad
\label{assumNewt}
\begin{itemize}
\item The Newton map is $k$-renormalizable around the free  critical point $x_0=0$, of filled Julia set denoted by $K_\lambda$\,;
\item The filled Julia set $K_\lambda$ intersects $\partial B_1$\,;
\item The filled Julia set $K_\lambda$  is locally connected.
\end{itemize}
\end{assu}

Note that since the Newton map is renormalizable, the critical point $x_0$ belongs to $K_\l$ so it cannot be on the closure of the immediate basin of attraction of the roots.
The following lemmas are proved in~\cite{RoeschAnnals}.
\begin{lem} The intersection $\partial B_1\cap K_\lambda$ reduces to exactly one point called  $\beta_\lambda$. There exists  a $k$-periodic angle $t_0$ such that
$R_1(t_0)$ lands on  at $\beta_\lambda$.
\end{lem}
\begin{lem}
The ray  $R'_1(2t_0)$ lands at $\beta'_\lambda= \partial W_1\cap K_\lambda$.\end{lem}

\begin{dfn} Under the assumption~\ref{assumNewt} we can define the following graphs.
  $$\Gamma_\lambda:=\Cup _{j=1}^{3}(\overline{ R_j(0)} \cup \overline{ R_j(1/2)} )\Cup _{j\ge 0} f^j(\Gamma') \hbox{ where }\Gamma'= R_{B'_1}(0)\cup R_{B'_1}(t_0)\cup K_\lambda.$$
\end{dfn}

Using the same proof as in   Lemma~\ref{component}  we get
\begin{lem} The set $\Gamma_\l$ is connected. Moreover, the set $\C \setminus \Gamma_\l$ is a union of three open connected and simply connected components.
\end{lem}
\begin{dfn}
Denote by $\Delta^{N_\l}_0$, $\Delta^{N_\l}_1$, $\Delta^{N_\l}_2$ the connected component of $\widehat \C\setminus \Gamma_\lambda$ such that the boundary of $\Delta^{N_\l}_0$ and  $\Delta^{N_\l}_1$ intersect $K_\lambda$,   $\Delta^{N_\l}_0$ containing $B'_2$ and $\Delta^{N_\l}_2$ being disjoint from $B_3$.
\end{dfn}

\begin{dfn} For any sequence $(\epsilon_i)_{i\in \N}\in \{0,1,2\}^\N$ we define $\Delta^{N_\l}_{\epsilon_0\ldots \epsilon_n}$ by induction over the length of the sequence $\epsilon_0\ldots \epsilon_n$ by the relation $\Delta^{N_\l}_{\epsilon_0\ldots \epsilon_n} = N_{\lambda}^{-n}(  \Delta^{N_\l}_{\epsilon_n})\cap\Delta^{N}_{\epsilon_{0} \ldots \epsilon_{n-1}}$.
\end{dfn}

\begin{dfn}\label{d:puzzleNewt}
Let $$U_n=\hat{\C}\setminus N^{-n}(\cup_{i=1,2,3}\phi_i^{-1}(\{z\in \C\mid \vert z\vert <r\}))$$ with $r<1$ and $\phi_i$ is the B\"ottcher coordinate in $B_i$.
Define  the puzzle piece of finite itinerary $\epsilon_0 \dots \epsilon_n$ as $$P^{N_\l}_{\epsilon_0 \dots \epsilon_n}=U_n\cap \Delta^{N_\l}_{\epsilon_0\cdots \epsilon_n}.$$
\end{dfn}
\begin{lem} \label{Newtpuzzle}
Any nest of  puzzle pieces shrinks to a point.
\end{lem}

\begin{figure}[ht]
  \begin{center}
 \includegraphics[scale=0.52]{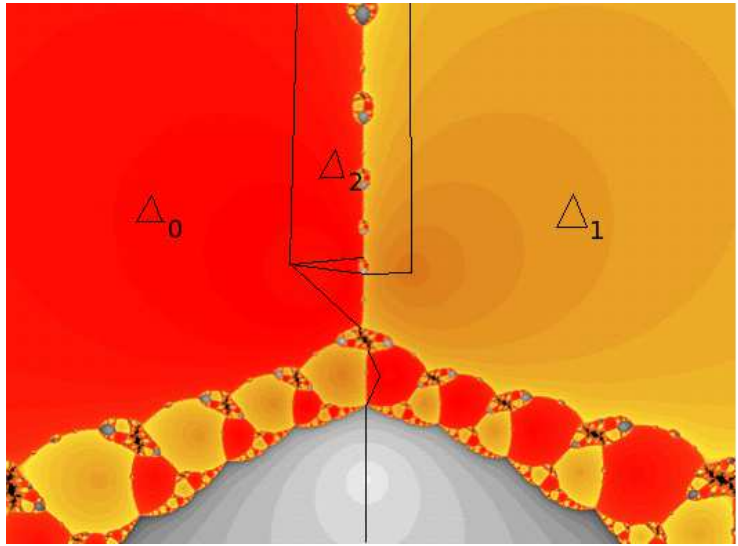}
  \end{center}
  \caption{The graph $\Gamma_\lambda$ and the sets $\Delta_j^{N_\lambda}$}
    \label{PartitionN}
\end{figure}

\proof The proof is similar to the proof of Lemma \ref{l:itifa}. Now, in this proof we will call the puzzle for the Newton map
the {\em original puzzle}, to distinguish it from a puzzle, which we will refer to as type $I$ or $II$, developed in Section 5 in
\cite{RoeschAnnals}. We want to consider a refinement of
the original puzzle with the puzzle of type $I$ or $II$ and use results in \cite{RoeschAnnals} to prove that indeed the
original puzzle also shrinks to points or iterated preimages of the small Julia set.

Let  $G(\theta)=\cup_{j\ge 0}N^j(R_1(\theta)\cup R_2(-\theta))$  be some graph for some periodic  angle $\theta$ with high period.
The type $I$ graph is just the union $$I(\theta)= \partial V_0\cup(G(\theta)\cup R_1(0)\cup R_2(0)\cup R_3(0))\cap V_0.$$
The type $II$ graph combine a graph of the type $G(\eta)$  together with the  so-called  articulated rays.
The articulated rays  are curves formed by an infinite sequence of closure of interior rays in the iterated pre-image of $B_1$ and $B_2$
(cf. axes of bubble rays in \cite{AY}).  Let $\gamma(\zeta)$ be such an  articulated ray (see Section 4 in \cite{RoeschAnnals}),
using  some periodic angle $\zeta$. Then  any point of this curve will eventually be mapped in the graph    $ G(\zeta)$.
The graph of type $II$ is the union $G(\zeta)\cup \bigcup_{j\ge 0}N_\l^j(\gamma(\zeta))\cup G(\eta)$.
It is forward invariant for $\eta $ and $\zeta$ periodic with high period. The puzzle pieces of level $n$ of type $I$ or $II$
defined by these  graphs are the connected components of the  complement of $N^{-n}(I(\theta))$ and $N^{-n}(II(\theta, \zeta))$
respectively (see also Definition 5.1 in \cite{RoeschAnnals}).

 \begin{figure}[ht]
  \begin{center}
 \includegraphics[scale=0.9]{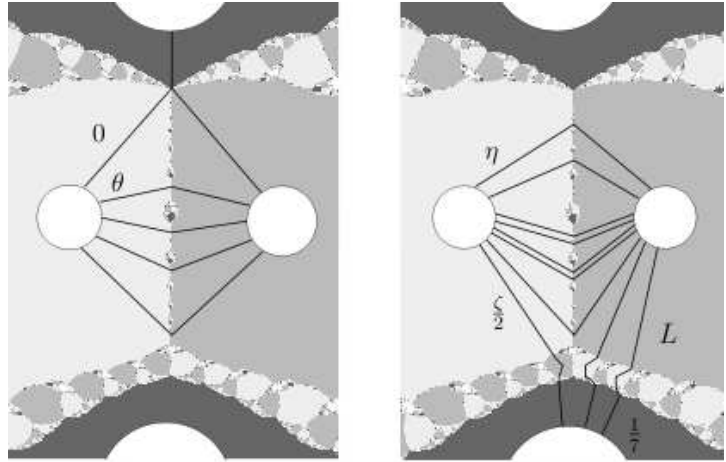}
  \end{center}
  \caption{Graphs of type $I$ and $II$ for the Newton method}
    \label{graphs}
\end{figure}

In~\cite{RoeschAnnals} (Proposition 8.13) it is  proven that  any nest  of such  puzzle pieces of type $I$ or type $II$ either shrinks to a point or
to  an iterated pre-image of the filled Julia set $K_\l$.
We now use this puzzle to show that any nest of puzzle pieces $P_{\epsilon_0,\ldots,\epsilon_n}$ also shrinks to points or iterated
preimages of the small Julia set.
First consider the refined puzzle consisting of all components in the complement of the union of the graphs $\Gamma_\l$ and the
graph of types $I$ and $II$. Clearly, any nest of puzzle pieces in this refined puzzle must shrink to points or subsets of
iterated preimages of the small Julia set according the results in \cite{RoeschAnnals} about the puzzle of type $I$ or $II$.
Note that each refined puzzle piece is contained in a puzzle piece  $P_{\epsilon_0,\ldots,\epsilon_n}$. Since original pieces
$P_{\epsilon_0,\ldots,\epsilon_n}$ are simply connected and do not contain critical points the map
$N_\l^n$ on all original puzzle pieces $P_{\epsilon_0, \ldots, \epsilon_n}$ is univalent for all depths $n$.
It is clear that the number of refined puzzle pieces of depth $0$ inside some puzzle piece $P_{\epsilon_0}$ is bounded
by some constant $K < \infty$. Since $N_\l^n$ is injective on all puzzle pieces $P_{\epsilon_0,\ldots,\epsilon_n}$, the number
of refined puzzle pieces of depth $n$ inside any $P_{\epsilon_0,\ldots,\epsilon_n}$ at depth $n$ is at most $K$ for all depths $n$.
It follows that the nest $P_{\epsilon_0,\ldots,\epsilon_n}$ also shrinks to points or a (subset of) iterated preimages of the small
Julia set.

Assume that we are in the second case and that our nests of type $I$ and $II$ contains $K_\lambda$.
Then $\bigcap_{n\in \N} \ol{P^{N_\l}_{\epsilon_0\cdots \epsilon_n}}\subset K_\lambda$. We consider the straightening map
$\sigma$ defined in a neighborhood of $K_\lambda$ to a neighborhood of $K(P_c)$ conjugating $N_\l$ to $P_c$ for some $c$.
Then the image  $\sigma( \ol{P^{N_\l}_{\epsilon_0\cdots \epsilon_n}})$ will be included in a puzzle piece of level $n$ of the nest defined as follows.
Let $\phi_{Pc}$ be  the B\"ottcher coordinate at infinity  of  the quadratic polynomial $P_c$.
It is then clear that $$\phi_{Pc}(\sigma( \ol{P^{N_\l}_{\epsilon_0\cdots \epsilon_n}}))\subset G_{\epsilon_0\cdots \epsilon_n}$$
where   $$G_{\epsilon_0\cdots \epsilon_n}=
\{z\in \C\mid \arg(z)\in [t_n,t'_n], \ 1\le \vert z\vert \le R^{1/2^n}\} $$ for some dyadic $t_n,t'_n$ with $\vert t_n-t'_n\vert\le 1/2^n$ and  for some $R>1$.
 Since we assume that the Julia set of $P_c$ is locally connected, we can conclude
that these nests of puzzle pieces  for $P_c$ shrinks to a point and therefore since $\sigma$ is an homeomorphism, the corresponding nest of (original) puzzle pieces
shrinks to a point.
\cqfd

\begin{lem}\label{l:puzzleitin} Let  $\{\epsilon_i\}_{i=0}^\infty, \{\epsilon'_i\}_{i=0}^\infty$ be two sequences in the same class in $\Sigma$, in other words  $[\{\epsilon_i\}_{i=0}^\infty]=[\{\epsilon'_i\}_{i=0}^\infty]$.
Then we have the property that $$\bigcap_{n\ge 0} \overline{P^{N_\l}_{\epsilon_0\cdots\epsilon_n}}=
\bigcap_{n\ge 0} \overline{P^{N_\l}_{\epsilon'_0\cdots\epsilon'_n}}.$$
\end{lem}
\proof
First we prove it for $\overline 0 \sim \overline 2$ (equivalence relation in $\tilde{\Sigma}$). Observe that $\infty$ is a repelling fixed point and the preimage  by $N_\l$ of $\Delta^{N_\l}_0$  contained in $\Delta^{N_\l}_0$ meets $\infty$. Since puzzle pieces shrinks to points we must have that

$$\{\infty\} =  \bigcap_{n\ge 0} \overline{P^{N_\l}_{\underbrace{ 0 \ldots 0}_{\text{$n$ digits}}  }}. $$

By an analogous argument, also
$$\{\infty\} =  \bigcap_{n\ge 0} \overline{P^{N_\l}_{\underbrace{2 \ldots 2}_{\text{$n$ digits}}  }}.$$

To continue, first note that $\infty$ has two preimages other than itself. Let the preimage $p_{\infty}' \neq \infty$ of $\infty$ meet $B_3$ and let $p_{\infty}'' \neq \infty$ meet $B_1$ and $B_2$.
By inspection, $\Delta^{N_\l}_{10}$ and $\Delta^{N_\l}_{02}$ meet $p_{\infty}'$ and $\Delta^{N_\l}_{20}$ and $\Delta^{N_\l}_{12}$ meet $p_{\infty}''$. The preimage of the set $\Delta^{N_\l}_{\underbrace{1 0 \ldots 0}_{\text{$n+1$ digits}}  }$ and $\Delta^{N_\l}_{\underbrace{0 2 \ldots 2}_{\text{$n+1$ digits}}  }$ meeting
$p_{\infty}'$ again meet $p_{\infty}'$ so these nested sets shrink to the same point $p_{\infty}'$ and likewise the other two nested sequences $\Delta^{N_\l}_{ \underbrace{20 \ldots 0}_{\text{$n+1$ digits}}  }$ and $\Delta^{N_\l}_{ \underbrace{12 \ldots 2}_{\text{$n+1$ digits}}  }$ also shrink to the same point $p_{\infty}''$. In other words,
$$p_{\infty}' =  \bigcap_{n\ge 0} \overline{P^{N_\l}_{\underbrace{1 0 \ldots 0}_{\text{$n+1$ digits}}  }} = \bigcap_{n\ge 0} \overline{P^{N_\l}_{\underbrace{0 2 \ldots 2}_{\text{$n+1$ digits}}  }}, $$
and
$$p_{\infty}'' =  \bigcap_{n\ge 0} \overline{P^{N_\l}_{\underbrace{ 20 \ldots 0}_{\text{$n+1$ digits}}  }} = \bigcap_{n\ge 0} \overline{P^{N_\l}_{\underbrace{1 2 \ldots 2}_{\text{$n+1$ digits}}  }}. $$

Any $n$th preimage $z$ of $p_{\infty}'$ or $p_{\infty}''$ lies inside some $\Delta^{N_\l}_j$. Hence any such point $z$ meets $\Delta^{N_\l}_{\epsilon_0,\ldots,\epsilon_{n-1}, q}$ where $q$ belongs to $\{1\overline{0},0\overline{2}\}$ (if $z$ is a preimage of $p_{\infty}'$) or $\{2\overline{0},1\overline{2}\}$ (if $z$ is a preimage of $p_{\infty}''$).
\endproof
We can now  define itineraries for points of the Julia set $J(N_\lambda)=J_\lambda$ with respect to this partition, as for the cubic polynomials. Note that the points lying on the graph
are not  accessible by  external rays.

\begin{dfn}\label{d:itNewt}
  For any point $z$ in $J_\lambda$  we associate $\epsilon_N(z)$ which is   a collection of itineraries  in $\Sigma$  defined as follows
 $$[\{\epsilon_i\}_{i=0}^\infty]\in  \epsilon_N(z)\iff z= \bigcap_{n\ge 0} \overline{P^{N_\l}_{\epsilon_0\cdots\epsilon_n}}.$$

  \end{dfn}

\newpage
  \begin{lem}
  \phantom{.} \quad
\begin{enumerate}
\item To the point $\infty$ we get the  itineraries $\epsilon_N(z) = \{[\ol 0]=[\ol 2], [\ol 1]\}$.
\item The pre-images of $\infty$ other than $\infty$ itself, $p_{\infty}'$ and $p_{\infty}''$, have $3$ itineraries each by pull back; the one on the boundary of $B_3$ has itineraries $\epsilon_N(p_{\infty}') = \{ [0\ol 1], [0\ol 2]=[1\ol 0] \}$, and $\epsilon_N(p_{\infty}'') = \{ [2\ol 0]=[1\ol 2], [2\ol 1] \}$.
\item For any point $z \in \cup N_\lambda^{-n}(K_\lambda)$, there are at most two itineraries in $\epsilon_N(z)$. Let $\underline{\epsilon}=[(\epsilon_0 \cdots \epsilon_n\cdots)] \in \epsilon_N(z) $. Then
 the sets $\Delta^{N_\l}_{\epsilon_0\ldots \epsilon_n}$ are nested and $z\in \ol{\Delta^{N_\l}_{\epsilon_0\ldots \epsilon_n}}$.
\end{enumerate}
\end{lem}

\proof The first part goes exactly as in  the proof of the previous lemma:
  $$\{ \infty\} = \bigcap_{n\ge 0} \overline{P^{N_\l}_{\underbrace{1  \ldots 1}_{\text{$n$ digits}}  }}. $$

For the second part, note that  the sequences $1\overline 0$ and $  0\overline 2$ define the same itinerary in $\Sigma$ so that $[1\overline 0]=[0\ol 2]$. Now,  from the previous lemma, it follows that $[1\overline 0] $ lies in the itinerary class $\epsilon_N(p_{\infty}')$ and similarly $[2\overline 0 ]=[ \overline 2]$ lies in the itinerary class $\epsilon_N(p_{\infty}'')$.
The proof of the fact that also $[0 \overline 1]$ belongs to $\epsilon_N(p_{\infty}')$ and that $[2 \overline 1]$ belongs to $\epsilon_N(p_{\infty}'')$ is precisely the same argument as in the previous lemma, and therefore we leave it to the reader.

If a point $z \in K_{\lambda}$ then it is adjacent to at least one of $\Delta^{N_\l}_1$ or $\Delta^{N_\l}_0$, or both. It is clear that is cannot be adjacent to $\Delta^{N_\l}_2$. Hence there are at most two itineraries in $\epsilon_N(z)$. If $z$ belongs to another preimage of $K_{\lambda}$ then $z$ belongs to only one $\overline{\Delta^{N_\l}_j}$.
\cqfd

\section{Parameter planes}
\subsection{Newton parameter plane}\label{ParamNewton}

We describe  the parameter space of cubic Newton maps. A detailed study of this space  is given in~\cite{Roeschcras, Roesch-thesis}.
Since any cubic Newton map can be conjugated to a rational map of the form $N_\lambda(z)=z-\frac{P_\lambda(z)}{P_\lambda'(z)}$ where $P_\lambda(z)=(z+1/2-\lambda)(z+1/2+\lambda)(z-1)$ and  $\lambda\in \C\setminus \{-3/2,0,3/2\}$, our parameter  space is  $ \C\setminus \{-3/2,0,3/2\}$.
\vskip 1em
{\it Symmetries.}
The map  $N_\lambda$ is conformally conjugated to $N_{\lambda'}$  if and only if $\lambda'=s(\lambda)$ where $s$ is any element of the group 
 of M\"obius transformations permuting  the three points $\{-3/2,0,3/2\}$.
The fundamental domain for this group    is
 $$\Omega=\{ \lambda\in \C\setminus \{-3/2,0,3/2\}\quad \mid  \quad \vert \lambda-1/2\vert<1,  \  \vert \lambda+1/2\vert<1, \  \Im m( \lambda)>0\}.$$

\begin{figure}[ht]
  \begin{center}
 \includegraphics[scale=0.52]{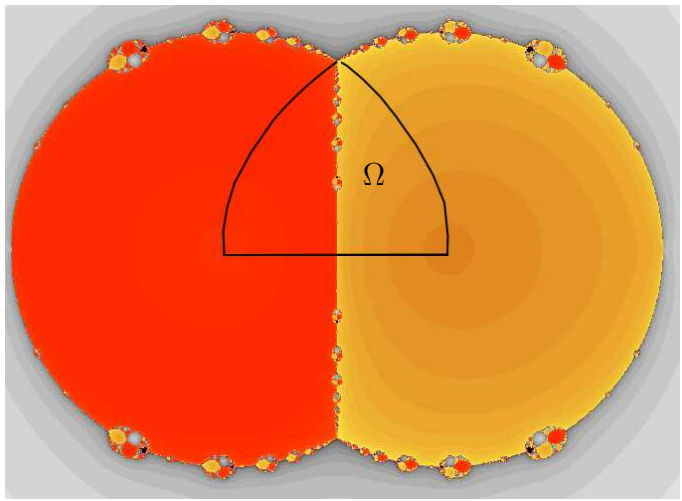}
  \end{center}
  \caption{Fundamental domain $\Omega$}
    \label{domainfunda}
    \vskip -14 em \hskip 2 em $\Omega$
\vskip 14 em

\end{figure}

Another symmetry is  given by  $\lambda\mapsto \ol{\lambda}$. Indeed,  $N_\lambda$ to $N_{\ol{\lambda}}$ are conjugated by $z\mapsto \ol{z}$. Therefore we restrict our   domain of study to  $$\Omega-=\Omega\cap \{z\mid \Re e(z)\le 0\}.$$

\vskip 1em
{\it Hyperbolic components.}  The roots of $P_\l$  $$Roots(\l)=\{-1/2+\l, -1/2+\l, 1\}$$are super attracting fixed points of $N_\l$. The hyperbolic  set $$\mathcal H=\{\l \in \C\mid N_\l^n(0)\to Roots(\l)\}$$
contains three principal hyperbolic components $\HH_-$, $\HH_+$   and $\HH_\infty$ around respectively  $-1/2$, $1/2$ and $\infty$.
The hyperbolic components $\HH_-$ and $\HH_+$ correspond through the map $\lambda\mapsto -\ol \lambda$ and  $\HH_-\subset \{\Re e(\lambda) \le 0\}$ as well as  $\HH_+\subset \{\Re e(\lambda) \ge 0\}$.

\vskip 1em
{\it Dynamics in $\Omega_-$.}
In $\Omega_-$, we call   $B_1$  the immediate basin of attraction of $p_1=-1/2-\lambda$, $B_2$ the one of $p_2=-1/2+\lambda$ and $B_3$ the one of $1$.
The boundaries  $\partial B_1$ and $\partial B_2$   intersect at the landing point of the rays  $R_1(0), R_2(0)$ and also at the landing point of the rays  $R_1(1/2),R_2(1/2) $.  The boundary  $\partial B_3$  intersects  $\partial B_1$ or $\partial B_2$  only at $\infty$ where  its ray $R_3(0)$  of angle $0$ lands. As a consequence of Lemmas \ref{threerays} and \ref{tworays} we get:

\begin{lem} For $\lambda \in \Omega_-$,  the set $\ol{R_1(0)\cup R_2(0)\cup R_1(1/2)\cup R_2(1/2)}$ is a Jordan curve and $t_0\in (0,1/2)$.
\end{lem}
The fact that $t_0\in (0,1/2)$ comes from   the position of the critical point  in the connected component of the  complement of the Jordan curve   containing $B_3$.

 \vskip 1em
 {\it Parametrizations.}
\begin{lem} There  exists a map $\Phi_-: \mathcal \HH_-\cap \Omega_-\to \D\setminus [0,1]$ which is a conformal bijection given by the position of the critical value in B\"ottcher coordinate.
\end{lem}
\begin{proof}[Sketch of proof] We give only the idea of the proof.  For more details see~\cite{Roeschcras, Roesch-thesis}. In  $\HH_-\cap\Omega_-$, the immediate basin of $p_1=-1/2-\lambda$ denoted by  $B_1$,  admits a unique B\"ottcher coordinate defined near
$P_1$ and  denoted by $\phi^\lambda_-$.
 This map is defined on some disk containing the critical point $0$ in its boundary whenever $0\in B_1$ and on   $B_1$ if $0\notin B_1$. Therefore the map $\Phi_-(\lambda)=\phi^\lambda_-(N_\lambda(0))$   is well defined for $\l\in \HH_-$.  This map defines a conformal bijection between
 $\HH_-\cap\Omega_-$ and $\D\setminus [0,1]$. 
 \end{proof}

Note that it extends to  a map from $\HH_-\cap\ol {\Omega_-}$ to $\D$, mapping the real
line to $[0,1]$ and the arc of circle also to $[0,1]$. Moreover this map extends to the closure of $\HH_-$  in $\ol{\Omega_-}$
because of the following result (see~\cite{Roeschcras} and \cite{Roesch-thesis}, Section 10.4) and Carath\'eodory's Theorem.

\begin{prop}   The boundary of $\mathcal H_-$ is a Jordan curve.
\end{prop}

By Carath\'eodory's Theorem we can extend $\Phi_-^{-1}$ to the boundary.
 This defines a map $\Phi_-^{-1}:\partial \HH_-\cap \ol{\Omega_-}\to \mathbb S $.

 \begin{dfn}  Let  the map  $\lambda:[0,1]\to \partial \HH_-\cap  \ol{\Omega_-}$
 given by   $\l(t)=\Phi_-^{-1}(e^{2i\pi t })$  be a parametrization of $\partial \HH_-\cap  \ol{\Omega_-}$.
 \end{dfn}

The following proposition is a consequence of \cite{Roesch-thesis}, (see Section 10.4, Lemma 7.5.1 and Corollary 7.6.10).
\begin{prop}  \label{Newtrays} Let $\l_0=\l(t)$ for some $t\in [0,1]$. We have the following dichotomy:
\begin{itemize}
\item  If $t/2$ is not periodic by multiplication by $2$ {\rm(}modulo $1${\rm)}, then the critical point $0$ is the
landing point of the ray $R_{1}(t/2)$\,; the map is not renormalizable around $0$\,;
\item If $t/2$ is $k$-periodic by multiplication by $2$,
then  the map is renormalizable around $0$ and $\l_0$ is the cusp of a copy of the Mandelbrot set noted $\M^N_t$.

Moreover, for any parameter $\l$ in $\M^N_t$ the map $N_\l$  is $k$-renormalizable
around $0$ with filled Julia set $K^\l$. The landing point of $R_1(t/2)$ is the    intersection  $\overline{B_1}\cap K^\l $.

\end{itemize}
\end{prop}

\begin{dfn}Let $T\subset [0,1]$ be the set of $t$ such that $t/2$ is $k$-periodic by multiplication by $2$ for some $k\ge 2$.
\end{dfn}

Recall that in the introduction we define $$RN= \bigcup_{t\in T} \M^N_t.$$

Any parameter $\l$ in $RN$  such that $J(N_\l)$ (or equivalently $K^\l$) is locally connected satisfies   Assumption~\ref{assumNewt}.

\subsection{Cubic polynomial parameter plane}
Any cubic polynomial having a critical fixed point is conjugate by an affine map to one in the family
$f_a(z)=z^2(z+3a/2)$.  Moreover, $f_a$ and $f_{a'}$ are affine conjugated if and only if $a'=-a$. Note  that the map $z\mapsto \ol z$ also conjugates $f_a$ to $f_{\ol a}$. Therefore, we can restrict ourself to study the maps in  one of the four quadrants
of $\C$. Let    $$Q=\{a\in \C\mid \Re e(a)>0, \Im m(a)<0\}$$ be the lower right quadrant. We will restrict to $Q$ in what follows because of the following lemma.

\begin{lem}
In the quadrant  $Q $, the rays $R_a^0(1/2)$ and $R_a^\infty(2/3)$ land at the same point.  
Moreover, $t_0\in(0,1/2)$.
\end{lem}
\begin{proof}
From the study of the real map in Lemma~\ref{real} we know that the critical point is on  the ray of angle $1/2$ of $A_1$ when $a\in \HH_0\cap \R^+$, in particular this ray crashes on the critical point.  When $a\notin \HH_0$, but $\Re e (a)>0$,  the rays $R_a^0(0)$ and $R_a^\infty(0)$ land at the same point. Therefore,  the ray $R^0_a(1/2)$
has to land at the same point as one pre-image $R^\infty_a(\theta)$  of  the ray $R_a^\infty(0)$ with
$\theta\in \{1/3,2/3\}$. The curve $\ol{R_a^0(0)\cup R_a^0(1/2)\cup R_a^\infty(0)\cup R_a^\infty(\theta)}$ separates the plane in two connected components. The one containing the critical point also contains the other ray pre-image
of $R_a^\infty(0)$ (because the width in the external B\"ottcher coordinate is greater than $1/3$ so the map cannot be injective). Now we fixed the quadrant $Q=\{a\in \C\mid \Re e(a)>0, \Im m(a)<0\}$. Since all the coordinates considered preserve the orientation, we deduce (by stability) that  critical point  belongs to the same component as the rays
$R^0_a(t)$ for $t\in (0,1/2)$. Thus the rays $R_a^0(1/2)$ and $R_a^\infty(2/3)$ land at the same point and the critical point belongs to the connected component of $\ol{R_a^0(0)\cup R_a^0(1/2)\cup R_a^\infty(0)\cup R_a^\infty(2/3)}$ containing $R_a^\infty(1/3)$.
\end{proof}

Note that  the two rays  $R_a^0(0)$ and $R_a^\infty(0)$ always land at the same point.

From \cite{RoeschENS} we get a parametrization of the hyperbolic component  $\HH_0$ which is the connected
component containing $0$ of $\HH=\{a\mid f^n_a(-a)\to 0\}$.
\begin{lem} There  exists a map $\Phi_0: \mathcal \HH_0\cap Q\to \D\setminus [0,1]$ which is a conformal bijection given by
the position of the critical value in B\"ottcher coordinate.
\end{lem}
\begin{proof}[Sketch of proof] We give the idea of the proof as follows. For details see \cite{RoeschENS}.
The immediate basin of $0$, denoted by  $A_1$,  admits a unique B\"ottcher coordinate defined near
$0$ and  denoted by
$\phi_a$. This map is defined on some disk containing the critical point $-a$ in its boundary whenever $-a\in A_1$ and on   $A_1$ if $-a\notin A_1$. Therefore the map $\Phi(a)=\phi_a(f_a(-a))$ is well defined for $a\in \HH_0$.  This map defines a conformal bijection between
$\HH_0\cap Q$ where $Q$ is any quadrant and $\D\setminus [0,1]$.
\end{proof}

\begin{figure}[h]
  \begin{center}
  \includegraphics[scale=0.52]{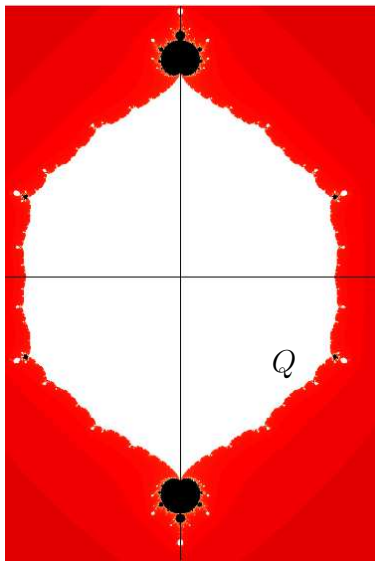}
 \end{center}
  \caption{Parameter space for $f_a$ and the quadrant $Q$}
    \label{graph2}
    \vskip - 10 em \hskip 6 em $Q$
\vskip 9 em
\end{figure}

As a consequence of \cite{RoeschENS}, we get the following three propositions (see Theorems 2, 3 and 5).
\begin{prop}   The boundary of $\mathcal H_0$ is a Jordan curve.
\end{prop}

By Carath\'eodory's Theorem we can extend $\Phi_0^{-1}$ to the boundary.
 This defines a map $\phi_0^{-1}:\partial \HH_0\cap Q\to \mathbb S \setminus\{1\}$.

 \begin{dfn}  Let  the map  $a:[0,1]\to \partial \HH_0\cap Q$   given by   $a(t)=\Phi_0^{-1}(e^{2i\pi t })$  be a parametrization of $\partial \HH_0\cap Q$.
 \end{dfn}
 \begin{prop}
Let $\mathcal L$ be a non empty connected component of $\mathcal C\setminus \overline{\HH_0}$. Then $\overline{\mathcal L} \cap \overline{\mathcal\HH_0}$ is only one point.
\end{prop}
\begin{dfn}
For $t \in [0,1]$, if $a(t)$ belongs to the closure of a connected component of $\overline{\mathcal L} \cap \overline{\mathcal\HH_0}$, we call this closed connected component $\mathcal L_t$, otherwise we define $\mathcal L_t$ to be $a(t)$.
In other words,
$$\mathcal C=\overline{ \mathcal\HH_0}\cup\bigsqcup_{t\in [0,1]} \mathcal L_t.$$
\end{dfn}

\begin{prop}  \label{Cubicrays} Let $a_0=a(t)$ for some $t\in [0,1]$. We have the following dichotomy:
\begin{itemize}
\item  If $t/2$ is not periodic by multiplication by $2$ {\rm(}modulo $1${\rm)}, then the critical point $-a_0$ is the landing point of the ray $R_{a_0}^0(t/2)$\,; the map is not renormalizable around $-a_0$\,;
\item If $t/2$ is $k$-periodic by multiplication by $2$,  then  the map is renormalizable around $-a_0$ and $a_0$ is the cusp of a copy of the Mandelbrot set noted $\M_t$.

 Moreover, for any parameter $a$ in $\M_t$ the map $f_a$  is $k$-renormalizable  around $-a$ with filled Julia set $K_{a}$. The landing point of $R_{a}^0(t/2)$ is the    intersection  $\overline{B_{a}}\cap K_a$.

\end{itemize}
\end{prop}

\begin{dfn}Let $T\subset [0,1]$ be the set of $t$ such that $t/2$ is $k$-periodic by multiplication by $2$ for some $k\ge 2$.
\end{dfn}

Recall that $$RC= \bigcup_{t\in T} \M_t.$$

Any parameter $a$ in $RC$  such that $J(f_a)$ is locally connected satisfies   Assumption~\ref{assumcubic}.
From~\cite{DR}, we know  that the Julia set $J(f_a)$ is locally connected   if and only if the small Julia set
$K_a$ is  locally connected. Hence parameters in $RC$ with  $K_a$
locally connected correspond to the parameters  satisfying assumption~\ref{assumcubic}.

\subsection{Correspondence between the parameter planes}
In this section we define a map $\NN$ from  $ RC\cap Q$ to  $  RN\cap \Omega_-$
for which the dynamics are similar. This map will be used in next section to construct the semi-conjugacies.

 For $a\in RC\cap Q$, the map $f_a$ is renormalizable and $a$ belongs to a copy $\M_t$ of $\M$ attached to $\partial \HH_0\cap Q$
 for some $t\in T$ (by previous section). For such $t\in T$,  the parameter $\lambda(t)$ in the Newton parameter plane  is a cusp of a copy $\M^N_t$ of the Mandelbrot set $\M$ (by section~\ref{ParamNewton}).

Denote by  $\chi_t: \M\to \M_t$, respectively  $\chi^N_t:\M \to \M^N_t$, the homeomorphisms between $\M$ and $\M_t$ or $\M^N_t$ respectively.
\begin{dfn}\label{mapN}Let $\NN:RC\cap Q\to RN\cap \Omega_-$ be defined by   $\NN(a)=\chi^N_t\circ \chi_t^{-1}(a)$.
\end{dfn}

\begin{rmk}
Using the symmetries  the map $\NN$ extends  to   $RC$ with image in   $ RN$.
Moreover, parameters in $RC$ with $K_a$ locally connected correspond  through the map $\NN$   to the parameters in $RN$  with $K^\l$ locally connected, {\it i.e.} those satisfying
assumption~\ref{assumNewt}.
\end{rmk}

\begin{lem} The map $\NN: RC \to RN$ is a bijection.
\end{lem}
\proof Indeed, we could have done the construction starting  in $RN\cap \Omega_-$. Let $\l\in RN\cap \Omega_-$, by definition, $\l$ belongs to  a copy $\M^N_t$ of $\M$ attached by $\l(t)$ to $\partial {\mathcal H_-}$ with $t\in T$. Then the parameter $a(t)$ is the cusp of a Mandelbrot copy $\M_t$ in the parameter plane of $f_a$  and $a=\chi_t\circ( \chi^N_t)^{-1}(\l)$. It belongs to $RC\cap Q$.
\endproof

\section{Construction of the semi-conjugacies}

Let $a\in RC\cap Q$ and $\lambda=\NN(a)$. The maps $f_a$ and $N_\l$ are renormalizable of same period $k\ge 2$ around the critical points $-a$ and $0$ respectively.   Denote by $K_a$ and $K^\l$  the small Julia sets  of $f_a^k$ and $N_\l^k$ respectively.

By the definition of  $\lambda$  there exists some $t \in T$ such that $\lambda= \NN(a)= \chi^N_t\circ \chi_t^{-1}(a)$ so that   $f_a^k$ is conjugate on some neighborhood of $K_a$  to  $P_c(z)=z^2+c$  where $c=\chi_t^{-1}(a)$ and   the map $N^k_\l$ is also   conjugate on some neighborhood of $K^\l$  to  the same $P_c$ because  $\lambda=\NN(a)= \chi^N_t(c)$. Moreover, $K_a\cap \partial A_1=\delta_a(t/2)$ and $K^\l\cap \partial B_1$ is the landing point of  the ray $R_1(t/2)$.

Denote,  by  $\sigma_a:U_a\to W$  the straightening map of $f^k_a$ defined on a neighborhood $U_a$ of $K_a$ onto a neigborhood $W$ of $K(P_c)$.
Similarly   let $\sigma^\l:U^\l\to W$ be the straightening map of $N_\l^k$ defined on a neighborhood  $U^\l$ of $K^\l$ onto $W$.

We suppose now that the small Julia sets $K_a$ and $K^\l$ are locally connected, so that Assumption~\ref{assumcubic} and Assumption~\ref{assumNewt} are satisfied.

\begin{prop}\label{p:itineraryclass} Let $\epsilon, \epsilon'$ be itineraries in $\Sigma$.
Then, $\epsilon, \epsilon'$ belong  to the same itinerary class for $f_a$ or for $f_\dbas$    if and only if they
 belong  to the same itinerary class for $N_\l$, {\it i.e.},  $\bigcap_{n\in \N} \overline{P^{N_\l}_{\epsilon_0\cdots \epsilon_n}}=\bigcap_{n\in \N} \overline{P^{N_\l}_{\epsilon'_0\cdots \epsilon'_n}}$.
\end{prop}

First recall that by Lemma~\ref{l:puzzleitin},  if two sequences
$(\epsilon_0\cdots \epsilon_n)$ and $(\epsilon'_0\cdots \epsilon'_n)$
define the same itinerary in $\Sigma $ (not itinerary class for $f_a$ nor $f_\dbas$ for the moment!)
then $$\bigcap_{n\in \N} \ol{P^{N_\l}_{\epsilon_0\cdots \epsilon_n}}=\bigcap_{n\in \N} \ol{P^{N_\l}_{\epsilon'_0\cdots \epsilon'_n}}.$$

 Proposition~\ref{p:itineraryclass} follows from the lemmas below.

Let $$\Lambda_a^0=  \overline{A'_1}\cup \overline{A_1}\cup \bigcup_{0\le n\le k}f_a^{n}( K_a),\quad    \Lambda_a^n= f_a^{-n}(\Lambda_a^0), \quad \Lambda_a^\infty=\bigcup_{n\ge 0}\Lambda_a^n.$$

Let $$\Lambda_\l^0=   \overline{W_1}\cup  \overline{B_1}\cup \bigcup_{0\le n\le k}N_\l^{n}( K^\l),\quad    \Lambda_\l^n= N_\l^{-n}(\Lambda_\l^0), \quad \Lambda_\l^\infty=\bigcup_{n\ge 0}\Lambda_\l^n.$$

\noindent
\begin{lem}\label{l:conjugacy} \phantom{.} \quad
\begin{enumerate} 
\item There exists a homeomorphism  $\psi_a:\Lambda_a^\infty\to \Lambda_\l^\infty$ which is a conjugacy   between   $f_a $ and $N_\lambda$ satisfying the following: For any itinerary  $\epsilon$,
$$z\in \Lambda_a^\infty\cap (\bigcap_{n\in \N} \overline{P^a_{\epsilon_0\cdots \epsilon_n}})  \iff \psi_a( z)\in \Lambda_\l^\infty\cap (\bigcap_{n\in \N} \overline{P^{N_\l}_{\epsilon_0\cdots \epsilon_n}}) \quad \forall n\ge 0.$$
\item
There exists a homeomorphism $\psi_\dbas:\bigcup_{n\ge 0}f_\dbas^{-n}(\ol A_2\cup \ol A_3) \to \bigcup_{n\ge 0}N_\lambda^{-n}(\ol B_2\cup \ol B_3),$  which is a conjugacy   between  $f_\dbas$ and $N_\lambda$ satisfying the following: For any itinerary  $\epsilon$, 
$$z\in \bigcup_{n\ge 0}f_\dbas^{-n}(\ol A_2\cup \ol A_3)\cap (\bigcap_{n\in \N} \overline{P^\dbas_{\epsilon_0\cdots \epsilon_n}})  \iff \psi_\dbas( z)\in \bigcup_{n\ge 0}N_\lambda^{-n}(\ol B_2\cup \ol B_3)\cap (\bigcap_{n\in \N}  \overline{P^{N_\l}_{\epsilon_0\cdots \epsilon_n}}) \quad \forall n\ge 0.$$
\end{enumerate}

\end{lem}
\proof  We explain the proof for $f_a$ since  is goes similarly   (and easier) for $f_\dbas$. In a first step we define $\psi_a$ on $\Lambda_a^0$.
Using that the boundary of  $A_1$ (the immediate  basin of attraction for the polynomial  $f_a$) and  of $B_1 $ (for the Newton map) are Jordan curves, we have  extended the B\"ottcher coordinates to the closure of these basins. The composition of these extended  maps gives  the desired   homeomorphims   $\psi_a: \overline{A_1} \rightarrow \overline{B_1}$,   it is a  conjugacy between $f_a$ and $N_\l$.

Now, we  extend $\psi_a$   to $K_a$ using the  straightening maps. Let $\psi_a(z)=(\sigma^\l)^{-1}\circ \sigma_a(z)$, it is defined on a neighborhood of $K_a$ and conjugates the maps $f_a$ and $N_\l$ on this neighbourhood. It defines a homeomorphism   between $K_a$ and $K^\l$ which  agrees with $\psi_a$ on $\partial A_1$  because  $K_a\cap \partial A_1=\delta_a(t/2)$ and $\beta_\l=K^\l\cap \partial B_1$, which is the landing point of the ray $R_1(t/2)$. From the formula $\psi_a(f_a(z))=N_\l(\psi_a(z))$ we extend the conjugacy $\psi_a$ on the forward images  $\cup_{0\le n\le k}f_a^{n}(\ol A_1\cup K_a)$. To extend $\psi_a$ ´to $A'_1$  there is no ambiguity since $A'_1$  is the  only  preimage of $A_1$ and the same holds for $N_\l$, i.e. $W_1$ is the only preimage of $B_1$.
The map $\psi_a$ is clearly an homeomorphism from  $\Lambda_a^0$  to $\Lambda_\l^0$ and it defines  a conjugacy between $f_a$ and $N_\l$.

An important remark is  that, by construction,  a point  $z\in \Lambda_a^0$   belonging to $\overline \Delta_j^a$ has its image $\psi_a(z)$ in  $\overline \Delta_j^\l$, and vice versa.

The second step  now is to extend $\psi_a$  by induction on  $\Lambda_a^\infty$.
The third (and last) step will be  verify  the property of $\psi_a$ on  $\Lambda_a^\infty\cap (\bigcap_{n\in \N} \overline{P^a_{\epsilon_0\cdots \epsilon_n}})$ (still by induction).

 We  want to extend $\psi_a$   by pull back  using the conjugacy formula
$N_\l(\psi_a(z))=\psi_a(f_a(z))$. Since neither $N_\l$ nor $f_a$ is injective,
we should precise which preimage  of $\psi_a(f_a(z))$  under  $N_\l$ we associate to $z$.
This is done by induction.  Note that the map $f_a$ is injective in any of the $\Delta_j^a$ since $\Delta_j^a$ is simply connected and does not contain  critical points.  Thus,   any point in $K(f_a)\setminus f_a(\Lambda_a^0)$ has at most  one preimage in each $\Delta_j^a$. Moreover the three preimages have to be in $\cup_{j=1,2,3}\Delta_j^a$ (because the graph $\Gamma_a$ is forward invariant) hence there is exactly one preimage in each $\Delta_j^a$.  The same holds for $N_\l$ and $\Delta_j^\l$.
Now assume that $\psi_a$ is defined on $\Lambda_a^n$ as a continuous conjugacy between $f_a$ and $N_\l$.
Then any point $z\in \Lambda_a^{n+1}\setminus \Lambda_a^n$  belongs to some $\Delta_a^j$, so we define   $\psi_a(z)$ as the preimage by $N_\l$ of $\psi_a(f_a(z))$ belonging  to $\Delta_\l^j$. This is possible since  $\psi_a(f_a(z))\notin \Lambda_\l^0$ because $f_a(z)\notin \Lambda_a^0$ (using the bijection $\psi_a$ on  $\Lambda_a^0$).

We prove the continuity  of $\psi_a$ by induction: Let $z\in \Lambda_a^{n+1}$, and $u=f_a(z) \in \Lambda_a^n$. The map $f_a$ is continuous   from a neigbourhood of $z$ to a neighborhood of $u$, the conjugacy $\psi_a$ is continuous on $\Lambda_a^n$ on a neighborhood of $u$ (by induction), and finally $N_\l$ is a homeomorphism from a neighbourhood of $\psi_a(z)$ to a neighborhood of $\psi_a(u)$ (since the critical points are on the graph). So the continuity of $\psi_a$ at $z$, follows by composition.

Now in the last step we verify that
$$z\in \Lambda_a^\infty\cap (\bigcap_{n\in \N} \overline{P^a_{\epsilon_0\cdots \epsilon_n}})  \iff \psi_a( z)\in \Lambda_\l^\infty\cap (\bigcap_{n\in \N} \overline{P^{N_\l}_{\epsilon_0\cdots \epsilon_n}})\quad  \forall n\ge 0.$$
Note that  the points considered are in the Julia set so  we need only to check now that
$$z\in \Lambda_a^\infty\cap (\bigcap_{n\in \N} \ol{ \Delta^a_{\epsilon_0\cdots \epsilon_n}})  \iff \psi_a( z)\in \Lambda_\l^\infty\cap (\bigcap_{n\in \N} \ol{ \Delta^{N_\l}_{\epsilon_0\cdots \epsilon_n}})\quad  \forall n\ge 0.$$

By the previous  construction of $\psi_a$ on $\Lambda_a^\infty$, we have chosen
$\psi_a(z)$ such that $z$ and $\psi_a(z)$  belong respectively  to  $\ol \Delta_j^a$ and
$\ol \Delta^{N_\l}_j$   with the same $j$. Therefore,  $z\in \Lambda_a^\infty\cap \ol \Delta_j^a$   if and only if   $\psi_a(z) \in  \Lambda_\l^\infty\cap \ol \Delta^{N_\l}_j$.
  Recall the definition of $\Delta_{\epsilon_0\cdots \epsilon_n}$:
 $$\Delta^a_ {\epsilon_0\cdots \epsilon_n}= f_a^{-n}(\Delta^a_{\epsilon_n})\cap \Delta^a_ {\epsilon_0\cdots \epsilon_{n-1}}\quad \hbox{ and } \Delta^{N_\l}_ {\epsilon_0\cdots \epsilon_n}= N_\l^{-n}(\Delta^{N_\l}_{\epsilon_n})\cap \Delta^{N_\l}_ {\epsilon_0\cdots \epsilon_{n-1}}$$

Now we have,
 \begin{align} z\in \Lambda_a^\infty\cap( \bigcap_{n\ge 0}  \ol {P^a_{\epsilon_0\cdots \epsilon_n}})  &\iff z\in \Lambda_a^\infty \cap \ol {P^a_{\epsilon_0\cdots \epsilon_n}}  \quad \forall n\ge 0
 \nonumber \\
&\iff z\in \Lambda_a^\infty \cap \ol {\Delta^a_{\epsilon_0\cdots \epsilon_n}}  \quad \forall n\ge 0
 \nonumber \\
 &\iff \forall n\ge 0, \  f_a^n(z)\in  \Lambda_a^\infty\cap \ol \Delta_{\epsilon_n}^a \hbox{ and }  z\in \Lambda_a^\infty \cap \ol {\Delta^a_{\epsilon_0\cdots \epsilon_{n-1}}}   \nonumber\\
& \iff
 \forall n\ge 0, \  \psi_a(f_a^n(z))\in  \Lambda_{\l}^\infty\cap
 \ol \Delta^{N_\l}_{\epsilon_n}  \hbox{ and } \psi_a(z)\in \Lambda_\l^\infty \cap \ol {\Delta^{N_\l}_{\epsilon_0\cdots \epsilon_{n-1}}}.
 \nonumber
\end{align}
Using the conjugacy relation $\psi_a(f_a^l(z))=N_\l^l(\psi_a(z))$ the last statement is equivalent to
 \begin{align}
 \forall n\ge 0,\  N_\l^n( \psi_a(z))\in  \Lambda_\l^\infty\cap \ol \Delta^{N_\l}_{\epsilon_n} &\hbox{ and } \psi_a(z)\in \Lambda_\l^\infty \cap
 \ol {\Delta^{N_\l}_{\epsilon_0\cdots \epsilon_{n-1}}} \nonumber \\
 &\iff  \psi_a(z)\in \Lambda_\l^\infty \cap \ol {\Delta^{N_\l}_{\epsilon_0\cdots \epsilon_n}}  \quad \forall n\ge 0
 \nonumber \\
  &\iff \psi_a( z)\in \Lambda_\l^\infty\cap (\bigcap_{n\in \N} \overline{P^{N_\l}_{\epsilon_0\cdots \epsilon_n}}).\nonumber
  \end{align}
\cqfd

{\bf Proof of proposition~\ref{p:itineraryclass}:} Proof of ($\Rightarrow$)
Take  two different itineraries  $\epsilon$ and $\epsilon'$,  which are in the same itinerary class for $f_a$: there exists $z\in J(f_a)$ such that $\epsilon$ and $\epsilon'$ belong to $\epsilon_a(z)$. Then by Corollary~\ref{c:biacc}  the angles $\theta=\theta(\epsilon)\neq\theta'=\theta(\epsilon')$    define rays landing at $z$, so
 the  point  $z$  is multiply accessible. This is only possible if  $z$ belongs to a preimage of $K_a$ by Lemma~\ref{l:multiaccess}.  So $z\in \Lambda^\infty_a$ and we already proved in  Lemma \ref{l:conjugacy} that $\psi_{a}$ is a bijection   satisfying
$$z\in \Lambda_a^\infty\cap (\bigcap_{n\in \N} \overline{P^a_{\epsilon_0\cdots \epsilon_n}})  \iff \psi_a( z)\in \Lambda_\l^\infty\cap (\bigcap_{n\in \N} \overline{P^{N_\l}_{\epsilon_0\cdots \epsilon_n}}).$$ The proof for $f_\dbas$ is similar.

Now we prove the converse ($\Leftarrow$). Assume that  two different itineraries  $\epsilon$ and $\epsilon'$ are in the same itinerary
class for the Newton map $N_\l$: $$\bigcap_{n\ge 0} \overline{P^{N_\l}_{\epsilon_0\cdots \epsilon_n}}=\bigcap_{n\ge 0} \overline{P^{N_\l}_{\epsilon'_0\cdots \epsilon'_n}}.$$ Let $z$ be the point at the intersection. Since interior of different puzzle pieces are disjoint, the point $z$  (which belongs to $\ol{P^{N_\l}_{\epsilon_0\cdots \epsilon_n}}$ and $\ol{P^{N_\l}_{\epsilon'_0\cdots \epsilon'_n}}$) has to be at the boundary of these puzzle pieces. This means that $z$ belongs to some preimage of the graph $\Gamma_\l$.
   Then $z$ either belongs to either $\bigcup_{n\ge 0}N_\l^{-n}(\ol B_2\cup \ol B_3)$ or to $\Lambda_\l^\infty$.

First case: assume that $z\in \Lambda_\l^\infty$.  We already proved in  Lemma \ref{l:conjugacy} that $\psi_{a}$  is a bijection   satisfying
$$u\in \Lambda_a^\infty\cap (\bigcap_{n\in \N} \overline{P^a_{\epsilon_0\cdots \epsilon_n}})  \iff \psi_a( u)\in \Lambda_\l^\infty\cap (\bigcap_{n\in \N} \overline{P^{N_\l}_{\epsilon_0\cdots \epsilon_n}}).$$
So $\psi_a^{-1}(z)\in  \bigcap_{n\in \N} \overline{P^a_{\epsilon_0\cdots \epsilon_n}}$ and $\psi_a^{-1}(z)\in  \bigcap_{n\in \N} \overline{P^a_{\epsilon'_0\cdots \epsilon'_n}}$. Then Lemma~\ref{l:itifa} implies that $\epsilon, \epsilon' \in  \epsilon_a(\psi_a^{-1}(z))$  and so  $\epsilon$ and $\epsilon'$  are in the same itinerary class for $f_a$.

The second case is similar: Assume that  that $z\in \bigcup_{n\ge 0}N_\l^{-n}(\ol B_2\cup \ol B_3)$.
Again Lemma \ref{l:conjugacy} gives that $\psi_{\dbas}$  is a bijection   satisfying
$$w\in \bigcup_{n\ge 0}f_\dbas^{-n}(\ol A_2\cup \ol A_3)  \cap (\bigcap_{n\in \N} \overline{P^\dbas_{\epsilon_0\cdots \epsilon_n}})  \iff \psi_\dbas( w)\in \bigcup_{n\ge 0}N_\l^{-n}(\ol B_2\cup \ol B_3)\cap (\bigcap_{n\in \N} \overline{P^{N_\l}_{\epsilon_0\cdots \epsilon_n}}).$$ The proof finishes as in the previous case using Lemma~\ref{l:itidbas}. \cqfd

\begin{lem}\label{l:semiconjugacy} The conjugacies $\psi_a$ and  $\psi_\dbas$ extend  to semi-conjugacies defined on  $K(f_a)$ and
 on $K_\dbas$. Moreover, they are conformal in the interior  of $K(f_a)$ and  $K_\dbas$ respectively.
\end{lem}
\proof We explain the extension for $\psi_a$, it is similar for $\psi_\dbas$.
We define the extension of $\psi_a$ as follows.
Let $z$ be a point of $K(f_a)$. It is in the closure of a decreasing sequence of puzzle pieces:
$z\in \ol{P^a_{\epsilon_0\cdots \epsilon_n}}$. Then we can extend previous definition by taking $\psi_a (z)$ to be   $$\displaystyle \psi_a (z):=\bigcap_{n\in \N} \ol{P^{N_\l}_{\epsilon_0\cdots \epsilon_n}}.$$

This makes $\psi_a$ well defined because if there are two different itineraries  $\epsilon$ and $\epsilon'$,  which are in the same itinerary class, then the nest of pieces  $\ol{P_{\epsilon_0\cdots \epsilon_n}^{a}}$ and $\ol{P_{\epsilon'_0\cdots \epsilon'_n}^{a}}$ shrink to the same point $z$.  Then by Proposition~\ref{p:itineraryclass} the two intersections $\bigcap_{n\in \N} \ol{P^{N_\l}_{\epsilon_0\cdots \epsilon_n}}$ and $\bigcap_{n\in \N} \ol{P^{N_\l}_{\epsilon'_0\cdots \epsilon'_n}}$ coincide.

With this definition of $\psi_a$,  we still have the property  $$z\in \ol{P^a_{\epsilon_0\cdots \epsilon_n}} \iff \psi_a(z)\in \ol{P^{N_\l}_{\epsilon_0\cdots \epsilon_n}}$$

We prove now the continuity of $\psi_a$.
Take  $z_0 \in K(f_a)$ with   $\psi_a(z_0) = \zeta_0$ and let $\alpha  > 0$ be given.  For $\epsilon=\epsilon(z_0)$, we have that the nest $\ol{P_{\epsilon_0\cdots \epsilon_n}^{a}}$ dicreases to the point $z_0$.

If  $z_0$ lies in the interior of the puzzle pieces: $z_0\in P_{\epsilon_0\cdots \epsilon_n}^{a}$ then   $\zeta_0$ lies in the interior of the puzzle pieces  $P^{N_\l}_{\epsilon_0\cdots \epsilon_n}$
since $\psi_a$ map the graph $\Gamma_a$ to the graph $\Gamma_\l$.   The sequence  $P^{N_\l}_{\epsilon_0\cdots \epsilon_n}$ shrinks to the point $\zeta_0$, so  for  some $n_0>0$ we have that  $P^{N_\l}_{\epsilon_0\cdots \epsilon_{n_0}}  \subset B(\zeta_0,\alpha)$ 
For this $n_0$ there exists $\delta > 0$ such that $B(z_0,\delta) \subset  P^a_{\epsilon_0\cdots \epsilon_{n_0}}$.
Therefore, if  $z$  satisfies $|z-z_0| < \delta$  then  we have $|\psi_a(z)-\psi_a(z_0)| < \alpha$. This proves the continuity at points which are not on the graph.

Now assume that  $z_0$ belongs to the boundary of the  puzzle pieces defining its nest.   There are finitely many such nests $P_{\epsilon^1}^a, P_{\epsilon^2}^a, \ldots, P_{\epsilon^k}^a$, which all shrink to $z_0$ (note that this can only happen if $z_0$ is in a preimage of $K_a$). Let $Q_n$ be the union of the closure of the puzzle pieces in the nests up to time $n$, i.e
\[
Q_n^a = \overline{P}_{\epsilon_0^1, \ldots, \epsilon_n^1}^a \cup \ldots \cup \overline{P}_{\epsilon_0^k, \ldots, \epsilon_n^k}^a.
\]
Let $Q_n^\l$ be the corresponding union   for the Newton map. Of course, $z_0$ lies in the interior of $Q_n^a $ and  $\zeta_0 $ lies in the interior of $ Q_n^\l$.
Now apply the same argument as before; i.e. let $\alpha > 0$ be given. Choose $n_0$ so large so that $Q_{n_0}^\l \subset B(\zeta_0,\alpha)$ and let $\delta > 0$ satisfy $B(z_0,\delta) \subset Q_{n_0}^a$. Then $|z-z_0| < \delta$ implies $|\psi_a(z)-\psi_a(z_0)| < \alpha$. This proves that $\psi_a$ is continuous.

Note that $K(f_{a})=\ol{\bigcup_{n\ge 0}\Gamma_n^a}$.
By construction,  $\psi_a$ is a conjugacy on $\Gamma_n^a$ for all $n$ so on the union  $\bigcup_{n\ge 0}\Gamma_n^a$. Now, by continuity,  the map $\psi_{a}$ is still a conjugacy on the closure so on $K(f_a)$.

 The map $\psi_a$ has been defined in $A_1$ using the B\"ottcher coordinate, so it is conformal in $A_1$. Now, defined by pullback, $\psi_a$ is also conformal in the preimages of $A_1$ since the maps $f_a$ and $N_\l$ are conformal on the preimages of $A_1$ and $B_1$ respectively.
 Any other connected component of the interior of $K(f_a)$ has to be a Fatou component in a preimage of $K_a$.  But the straightening map $\sigma_a$ is conformal in the interior of $K_a$ and similarly for
 $\sigma_\l$. Therefore, $\psi_a$ is conformal on the interior of $K_a$ and by pullback on the interior of any preimage of $K_a$. The result follows.
\cqfd

\section{Ray equivalence}
First recall that for the map $f_\dbas$, if the ray $R_\dbas^\infty(t)$ lands at a point $w$ then $\epsilon(-t)\in \epsilon_\dbas(w)$  whereas for the map $f_a$, if the ray $R_a^\infty(t)$ lands at a point $z$ then $\epsilon(t)\in \epsilon_a(z)$.

The ray equivalence relation $\sim_r$ then can be express with itinerary classes   as follows:

For  two points $u,v$ in $J(f_a)\cup J(f_\dbas)$,  the relation   $u \sim_r v $ means  that there exist angles  $t_1, \ldots, t_n$, points $u=z_0,z_1, \ldots, z_{n-1},z_n=v$ in  $J(f_a)\cup J(f_\dbas)$  and $\alpha, \beta$ in $\{a, \dbas\}$ such that
\begin{align}\label{raychain}
&\epsilon(t_1) \in \epsilon_{\alpha_0}(z_0),   \\
& \epsilon(t_1), \epsilon(t_2) \in \epsilon_{\alpha_1}(z_1), \nonumber \\
& \epsilon(t_2), \epsilon(t_3) \in \epsilon_{\alpha_2}(z_2), \nonumber \\
&... \nonumber \\
& \epsilon(t_{n-1}), \epsilon(t_n) \in \epsilon_{\alpha_{n-1}}(z_{n-1}), \nonumber \\
& \epsilon(t_n) \in \epsilon_{\alpha_{n}}(z_{n}), \nonumber
\end{align}
where $\alpha_k = \alpha$ for even $k$ and $\alpha_k = \beta$ for odd $k$.

The following lemma follows from the definition of itinerary classes.
\begin{lem}\label{l:lemma1}
Let   $u \in J(f_{\alpha})$ and $v \in J(f_{\beta})$. Then,
  \begin{align}\epsilon_{\alpha}(z) \cap \epsilon_{\beta}(w) \neq \emptyset \iff& \exists t \in [0,1], \hbox{ such that }  \epsilon(t) \in \epsilon_{\alpha}(u) \cap \epsilon_{\beta}(v) \nonumber \\ \iff &\exists t \in [0,1], \hbox{ such that }  \partial R_{\alpha}^{\infty}(t) = \gamma_{\alpha}(t) = u, R_{\beta}^{\infty}(-t) = \gamma_{\beta}(-t) = v . \nonumber \end{align}

In particular, $\epsilon_{\alpha}(u) \cap \epsilon_{\beta}(v) \neq \emptyset$ implies that $u \sim_r v$.
\end{lem}
Note however that the converse of the last statement if not true; $u \sim_r v$ does not necessarily imply that $\epsilon_{\alpha}(u) \cap \epsilon_{\beta}(v) \neq \emptyset$.

\begin{lem}\label{l:lemma2}
Suppose that $z \in J(f_{\alpha})$, $w \in J(f_{\beta})$ and that $ \epsilon_{\alpha}(z) \cap \epsilon_{\beta}(w) \neq \emptyset$. Then $\psi_{\alpha}(z) = \psi_{\beta}(w)$.
\end{lem}
\proof
By Proposition~\ref{p:itineraryclass} and  definition of $\psi_a$ and $\psi_\dbas$ of Lemma~\ref{l:semiconjugacy}, $\psi_{\alpha}(z) = \bigcap_{n\in \N} P_{\epsilon_0\cdots \epsilon_n}^{N_\l}$ for all $\epsilon=[\{\epsilon_0, \cdots, \epsilon_n\}] \in \epsilon_{\alpha}(z)$ and
$\psi_{\beta}(w) = \bigcap_{n\in \N} P_{\epsilon'_0\cdots \epsilon'_n}^{N_\l}$ for all $\epsilon' \in \epsilon_{\beta}(w)$.
Since $  \epsilon_{\alpha}(z) \cap \epsilon_{\beta}(w) \neq \emptyset$  there exists $\epsilon''$ in the intersection. So we must have that
\[
\bigcap_{n\in \N} P_{\epsilon_0\cdots \epsilon_n}^{N_\l} = \bigcap_{n\in \N} P_{\epsilon''_0\cdots \epsilon''_n}^{N_\l}= \bigcap_{n\in \N} P_{\epsilon'_0\cdots \epsilon'_n}^{N_\l}
\]
and hence $\psi_{\alpha}(z) = \psi_{\beta}(w)$.
\endproof

\begin{prop}\label{p:rayequivalence}
Suppose that $z \in J(f_{\alpha})$ and $w \in J(f_{\beta})$. Then $z \sim_r w$ if and only if $\psi_{\alpha}(z) = \psi_{\beta}(w)$.
\end{prop}

\begin{proof}
($\Leftarrow$) By definition
\[
\psi_{\alpha}(z) = \bigcap_{n\in \N} P_{\epsilon_0\cdots \epsilon_n}^{N_\l} = \bigcap_{n\in \N} P_{\epsilon'_0\cdots \epsilon'_n}^{N_\l} = \psi_{\beta}(w),
\]
for some itineraries $\epsilon \in \epsilon_{\alpha}(z)$ and $\epsilon' \in \epsilon_{\beta}(w)$. By Proposition 4.1
we have that either $\epsilon, \epsilon' \in \epsilon_{\alpha}(z)$ or $\epsilon, \epsilon' \in \epsilon_{\beta}(w)$. In both cases we have $\epsilon_{\alpha}(z) \cap \epsilon_{\beta}(w) \neq \emptyset$. Lemma~\ref{l:lemma1} now gives that $z \sim_r w$.

($\Rightarrow$) Now the chain relation  (\ref{raychain})  holds,  so for all $2 \leq k \leq n-1$ we have
\begin{align}
\epsilon(t_{k-1}),  \epsilon(t_{k}) \in \epsilon_{\alpha_{k-1}}(z_{k-1}) \nonumber  \\
\epsilon(t_{k}), \epsilon(t_{k+1}) \in \epsilon_{\alpha_{k}}(z_{k}). \nonumber
\end{align}
Obviously $\epsilon_{\alpha_{k-1}}(z_{k-1}) \cap \epsilon_{\alpha_k}(z_{k}) \neq \emptyset$ so Lemma~\ref{l:lemma2} gives that $\psi_{\alpha_{k-1}}(z_{k-1}) = \psi_{\alpha_k}(z_{k})$. Since this holds for all $2 \leq k \leq n-1$ we must have
$\psi_{\alpha}(z) = \psi_{\beta}(w)$.
\end{proof}

\section{Proof of the main theorem}

We have defined the sets $RC $ and $RN $ in the parameter plane section and the map $\NN$ in Definition~\ref{mapN}.
 We proved that it is a bijection. Now we shall prove Theorem \ref{resultat1} which we recall here:

 \begin{thmii}{\bf\ref{resultat1}}
For any parameter $a\in RC$ the polynomials  $f_a$ and $f_{\dbas}$ are conformally mateable if $J(f_a)$ is locally connected.
Moreover, $\NN(f_a)$ is the mating of  the  polynomials  $f_a$ and $f_{\dbas}$. 
\end{thmii}

\proof
Let us prove that $f_a$ and $f_\dbas$  are mateable for $a\in RC$ if  $J(f_a)$ is locally connected.  In Lemma \ref{l:semiconjugacy}  we constructed continuous maps \begin{eqnarray}
\psi_{\dbas}:& K(f_{\dbas}) \rightarrow \hat{\C} \\
\psi_a:& K(f_a) \rightarrow \hat{\C}.
\end{eqnarray}
which are   semi-conjugacies with $N_\l$ when  $\l=\NN(a)$, {\it i.e;}
\[
N_{\lambda} \circ \psi_{a} (z) = \psi_{a} \circ f_{a}(z) \hbox{ for } a\in K(f_a)\quad \hbox{ and }  N_{\lambda} \circ \psi_{\dbas} (z) = \psi_{\dbas} \circ f_{\dbas}(z) \hbox{ for }z \in K(f_\dbas).
\]
These semi-conjugacies are conformal on the interior or the filled in Julia sets.

Moreover, in Proposition~\ref{p:rayequivalence} we proved that
$$\forall (z,w)\in K(f_a)\times K(f_\dbas), \quad \psi_a(z)=\psi_\dbas(w)\iff z\sim_r w.$$

To finish, we should  prove that  $\psi_a( K(f_a))\cup \psi_\dbas( K(f_\dbas))=\widehat \C$.
Any point $u\in J(N_\l)$ has an itinerary class $\epsilon_N(u)$ with respect to the Newton map $N_\l$ (Definition~\ref{d:itNewt}): $$[\{\epsilon_i\}_{i=0}^\infty]\in  \epsilon_N(u)\iff z= \bigcap_{n\ge 0} \overline{P^{N_\l}_{\epsilon_0\cdots\epsilon_n}}.$$
Then let $$z= \bigcap_{n\ge 0} \overline{P^a_{\epsilon_0\cdots\epsilon_n}}\hbox{ and } w= \bigcap_{n\ge 0} \overline{P^\dbas_{\epsilon_0\cdots\epsilon_n}}.$$
By lemma~\ref{l:semiconjugacy},    $\psi_a(z)=u$ and $\psi_\dbas(w)=u$.

Now, any point $u\in \widehat\C\setminus J(N_\l)$ belongs to a Fatou component which is a preimage of either $B_1$, $B_2$, $B_3$ or a Fatou component in $K^\l$.

  If $u$ is in a preimage of $B_1 \cup K^\l$, then
 $u\in \Lambda_\l^\infty$.  We already proved in  Lemma \ref{l:conjugacy} that $\psi_{a}$  is a bijection   satisfying
$$x\in \Lambda_a^\infty\cap (\bigcap_{n\in \N} \overline{P^a_{\epsilon_0\cdots \epsilon_n}})  \iff \psi_a( x)\in \Lambda_\l^\infty\cap (\bigcap_{n\in \N} \overline{P^{N_\l}_{\epsilon_0\cdots \epsilon_n}}).$$
So we get a preimage $z$ of $u$ under $\psi_a$.

Now, if  $u$ is in a preimage of $B_2\cup B_3$,
as well  Lemma \ref{l:conjugacy} gives that $\psi_{\dbas}$  is a bijection   satisfying
$$x\in \bigcup_{n\ge 0}f_\dbas^{-n}(\ol A_2\cup \ol A_3)  \cap (\bigcap_{n\in \N} \overline{P^\dbas_{\epsilon_0\cdots \epsilon_n}})  \iff \psi_\dbas( x)\in \bigcup_{n\ge 0}N_\l^{-n}(\ol B_2\cup \ol B_3)\cap (\bigcap_{n\in \N} \overline{P^{N_\l}_{\epsilon_0\cdots \epsilon_n}}).$$ So we get a preimage $w$ of $u$ under $\psi_\dbas$.
This completes the proof that $\NN(f_a)$ is the mating of $f_a$ and $f_\dbas$.

\section{The boundary of $\HH_0$}
In this section we prove Theorem \ref{resultat2}. Recall that
$$NRC = \{ a \in \partial \HH_0 : \text{$a$ does not belong to a Mandelbrot copy} \}.$$ Then Theorem \ref{resultat2} will be a consequence of:
\begin{prop} \label{propH0} The map $\NN$ extends to $NRC$.
It defines a map from  $NRC$ to $\partial \Omega_-$ such that $\NN(f)$ is the mating of $f$ with $f_\dbas$.
\end{prop}
The proof is very similar to the previous Main Theorem \ref{resultat1}. To avoid lengthy repetitions of lemmas and propositions of
previous sections we give instead references to such results where the proofs are analoguous. The only major difference compared
to the Main Theorem is that $f_a$ is not anymore renormalizable around the critial point $-a$ for $a \in \partial \HH_0$
(not in a Mandelbrot copy) and therefore the graph for $f_a$ will look different; the small Julia set $K_a$ reduces to a point $-a$.
Moreover, $-a$ belongs to the boundary of $A_1$ and $A_1'$ meet $A_1$ precisely at $-a$. It also turns out that the original graphs for $f_a$ and $N_\l$ are not forward invariant anymore. We will slightly extend these graphs to make them forward invariant.

We define now the map $\NN$ in $\partial \HH_0\cap Q$.
\begin{dfn}
For $a\in \partial \HH_0\cap Q$, there exists some $t\in \R/\Z$ such that $a=a(t)$. Let $\NN(a)$ be the point $\lambda(t) \in \partial \HH_-$.
\end{dfn}
\begin{rmk} When $t\in T$, that is when $t/2$ is periodic by multiplication by $2$,
the point $a$ belongs to a copy of $\MM$ called $\MM_t$ and $\NN(a)$ as already been defined.
It is clear that the two definitions coincides.
\end{rmk}
Now we should only consider the case where $t\notin T$.
The following lemma is a consequence of Propositions \ref{Newtrays} and \ref{Cubicrays}.
\begin{lem}For $a\in \partial \HH_0\cap Q$, the critical point $-a$ is the landing point of the ray $R^a_0(t/2)$. For $\l=\l(t)=\NN(a)$,
the critical point $x_0=0$ is the landing point of the ray   $R_1(t/2)$.
\end{lem}

\begin{dfn} Consider the following graphs:
$$\Gamma_a=\overline{A_1} \cup \ol{R_a^0(t/2)}\cup \ol{R_a^0(1/2)}\cup \ol{R_a^0(0)}\cup\ol{ R_a^\infty(0)}\cup R_a^\infty (1/3)\cup\ol{ R_a^\infty(2/3)}\cup\ol{ R'_a(0)}\cup\ol{ R'_a(t)},$$
$$\Gamma_\l=\overline{B_1} \cup \ol{R_1(0)}\cup \ol{R_1(t/2)}\cup \ol{R_2(1/2)}\cup R_2(0)\cup\ol{ R_3(0)}\cup R_3 (1/2)\cup\ol{ R'_1(0)}\cup\ol{ R'_1(t)}\cup\ol{ R_1(2t)}.$$
\end{dfn}
Note that these graphs are the original graphs for $f_a$ and $N_\l$ with the small Julia sets replaced by the free critical point, and moreover, for $f_a$ we have added $\overline{A_1}$ and for $N_\l$ we have added $\overline{B_1}$ so that the graphs are still forward invariant. Hence we can define puzzle pieces analoguously.

We have the following by construction.
\begin{lem}Each  graph is connected and defines 3 connected  components in its complement.
\end{lem}

\begin{dfn} Let $\Delta_0^a$ , $\Delta_1^a$ , $\Delta_2^a$   denote the component  of $\C\setminus \Gamma_a$ containing the rays $R_a^{\infty}(t)$ for $t\in (0, 1/3)$, for $t\in ( 1/3, 2/3)$, for $t\in ( 2/3, 1)$ respectively.
Likewise, for the Newton map, let $\Delta_0^{N_\l}$ , $\Delta_1^{N_\l}$ , $\Delta_2^{N_\l}$   denote the component  of $\C\setminus \Gamma_\l$ containing the rays $R_3(t)$ for $t \in (1/2,1)$, the rays $R_3(t)$ for $t \in (0, t/2)$, and the rays $R_2(t)$ for $t\in (0, 1/2)$ respectively.  \end{dfn}

\begin{dfn} For any sequence $(\epsilon_i)_{i\in \N}\in \{0,1,2\}^\N$   we define $\Delta^a_{\epsilon_0\ldots \epsilon_n}$  by the relation $$\Delta^a_{\epsilon_0\ldots \epsilon_n} = f_{a}^{-n}(  \Delta^a_{\epsilon_n})\cap\Delta^{a}_{\epsilon_{0} \ldots \epsilon_{n-1}}.$$
Analogously, for any sequence $(\epsilon_i)_{i\in \N}\in \{0,1,2\}^\N$   we define $\Delta^{N_\l}_{\epsilon_0\ldots \epsilon_n}$  by the relation $$\Delta^{N_\l}_{\epsilon_0\ldots \epsilon_n} = {N_\l}^{-n}(  \Delta^{N_\l}_{\epsilon_n})\cap\Delta^{{N_\l}}_{\epsilon_{0} \ldots \epsilon_{n-1}}.$$
\end{dfn}

\begin{lem} For any sequence $(\epsilon_i)_{i\in \N}\in \{0,1,2\}^\N$   the intersection  $\bigcap_{n\ge 0}(\overline{\Delta^a_{\epsilon_0\ldots \epsilon_n}}\cap V_n)$  reduces to a point, as well as the intersection $\bigcap_{n\ge 0}(\overline{\Delta^{N_\l}_{\epsilon_0\ldots \epsilon_n}} \cap W_n)$.
\end{lem}
\proof The proof of the first part of the lemma follows an analogous argument to the proof of Lemma \ref{l:itifa}. No major changes are needed; we leave the details to the reader.

In the non-renormalizable case for the Newton map (i.e. when $N_\lambda$ is not renormalizable around its free critical point $0$)
an additional graph of type $III$ is constructed
(see Definition 8.13 in \cite{RoeschAnnals}) which is a union of a graph of type $I$ and type $II$. Analogously, puzzles pieces of type $III$
(defined by the graph of type $III$) also shrink to points or iterated preimages of the small Julia set.
To prove that the puzzle pieces $P_{\epsilon_0 \ldots \epsilon_n}^{N_\l}$ refined by a puzzle of type $III$ instead of type $I$ or
$II$ also shrink to points or iterated preimages of the small Julia set, precisely the same argument is used as in the proof of
Lemma \ref{Newtpuzzle}.
\endproof

We want to prove that this map $\NN(a)=N_\l$ is the mating of $f_a$ and $f_\dbas $.
\begin{lem}There exists an homeomorphism $\psi_a: \bigcup_{n\ge 0} f_a^{-n}(\ol {A_1})\to  \bigcup_{n\ge 0} N_\l^{-n}(\ol{B_1})$,
which is a conjugacy   between   $f_a $ and $N_\lambda$ satisfying the following: For any itinerary  $\epsilon$,
$$z\in \bigcup_{n\ge 0} f_a^{-n}(\ol {A_1})\cap (\bigcap_{n\in \N} \overline{\Delta^a_{\epsilon_0\cdots \epsilon_n}})  \iff \psi_a( z)\in \bigcup_{n\ge 0} N_\l^{-n}(\ol{B_1})\cap (\bigcap_{n\in \N} \overline{\Delta^{N_\l}_{\epsilon_0\cdots \epsilon_n}}) \quad \forall n\ge 0.$$
\end{lem}
\proof The proof mimics the proof of  Lemma~\ref{l:conjugacy}.
First, the map $\psi_a$ is defined by the extended B\"ottcher coordinates on $\ol {A_1}$ to $\ol{B_1}$. Then by pullback it is uniquely defined  because the critical point is on the boundary of $A_1$ (and $B_1$ for the Newton map).

If we then let
\begin{equation}
\Lambda_a^0 = \overline{A_1'} \cup \overline{A_1}  \text{ and } \Lambda_a^{\lambda} = \overline{B_1'} \cup \overline{B_1}
\end{equation}
in Lemma~\ref{l:conjugacy}, the proof goes through in the same way, apart from the fact that one do not need the construction of conjugacies between the small Julia sets $K_a$ and $K_\lambda$ (since they do not exist in this case).
 \endproof

Similar to Lemma~\ref{l:itfa} we have:
\begin{lem}\label{l:itfabis} For any  point $z\in J(f_a) $  we have $$[(\epsilon_i)_{i\in \mathbf N}]\in \epsilon_a(z)
\iff z\in \bigcap_{n\in \mathbf N} \overline{\Delta^a_{\epsilon_0\ldots \epsilon_n}}.$$
 \end{lem}

\begin{dfn} We say that two itineraries $\epsilon=[ \{\epsilon_n\}_{n\in \N}]$ and $\epsilon'=[ \{\epsilon'_n\}_{n\in \N}]$  are in the same itinerary class  for $N_\l$ if and only if  $$J(N_\l)\cap
\bigcap_{n\in \N} \overline{\Delta^{N_\l}_{\epsilon_0\cdots \epsilon_n}}=J(N_\l)\cap\bigcap_{n\in \N} \overline{\Delta^{N_\l}_{\epsilon'_0\cdots \epsilon'_n}}$$
\end{dfn}

\begin{prop}\label{p:itineraryclassbis} Let $\epsilon, \epsilon'$ be itineraries in $\Sigma$.
Then, $\epsilon, \epsilon'$ belong  to the same itinerary class for $f_a$ or for $f_\dbas$    if and only if they
 belong  to the same itinerary class for $N_\l$.
\end{prop}
\proof The proof goes as in the proof of Proposition~\ref{p:itineraryclass}. Replace $K_a$ with $-a$ (also in the graph $\Gamma_\l$) and use
\begin{equation}
\Lambda_a^0 = \overline{A_1'} \cup \overline{A_1}  \text{ and } \Lambda_a^{\lambda} = \overline{B_1'} \cup \overline{B_1}.
\end{equation}
Moreover, it is straightforward to see that the proof of Lemma \ref{l:itifa} (Lemma \ref{l:itifa} is used in the proof of the Proposition) goes
through analogously with the new definition of the graph $\Gamma_a$, when $a \in \HH_0$.
\endproof

\begin{lem}\label{l:semiconjugacybis} The conjugacies $\psi_a$  extends  to semi-conjugacy defined on  $K(f_a)$. Moreover, it is conformal in the interior  of $K(f_a)$.
\end{lem}
\proof  The proof is the same as in the proof of Lemma~\ref{l:semiconjugacy} except that we  have to consider the graphs defined in~\cite{RoeschENS} and in~\cite{RoeschAnnals}.\cqfd

Using the same ideas as in the proof of Proposition~\ref{p:rayequivalence} we get the following analogue of that Proposition. We omit the proof.
\begin{prop}\label{p:rayequivalencebis}
Suppose that $z \in J(f_{\alpha})$ and $w \in J(f_{\beta})$. Then $z \sim_r w$ if and only if $\psi_{\alpha}(z) = \psi_{\beta}(w)$.
\end{prop}

Again, following the proof of Theorem~\ref{resultat1}, we get the following which concludes the proof of Proposition \ref{propH0}, and from which Theorem \ref{resultat2} follows. 
 \begin{thmii}
For any parameter $a\in \partial \HH_0\cap Q$ the polynomials  $f_a$ and $f_{\dbas}$ are conformally mateable. Moreover, $\NN(f_a)$ is the mating of  the  polynomials  $f_a$ and $f_{\dbas}$. 
\end{thmii}


\newpage





















































 \begin{thebibliography}{1}

\frenchspacing
\bibitem[Ah]{A} {\sc L. V. Ahlfors} ---
{\em Lectures on quasi-conformal mappings},
Wadsworth~\&~Brook/Cole, Advanced Books \&~Software, Monterey
1987.

\bibitem[AsYa]{AY} {\sc M. Aspenberg, M. Yampolski}  ---
{\em Mating  non-renormalizable quadratic polynomials}, Comm. Math. Phys. {\bf 287} (2009), no.1, pp. 1--40.

\bibitem[Bl]{Bl} {\sc P. Blanchard} ---
             {\em  Complex analytic dynamics on the Riemann sphere},
              Bull. Amer. Math. Soc.  {\bf 11} (1984), pp. 85--141.


\bibitem[BrHu]{BH} {\sc B.~Branner} and {\sc J.~H.~Hubbard} ---
{\em The iteration of cubic polynomials, Part.1\,: The global
topology of the parameter space}, Acta. Math. {\bf 160} (1988),
pp. 143--206.

\bibitem[DeRo]{DR}{\sc A.~Dezotti, P. Roesch}
{\em On (non-)local connectivity of some Julia sets},   Frontiers in Complex Dynamics:
     In Celebration of John Milnor's 80th Birthday
     Edited by Araceli Bonifant, Misha Lyubich,  Scott Sutherland
Princeton University Press,  Princeton NJ, 2014, pp. 135--162.

\bibitem[Do1]{Do} {\sc A.~Douady} ---
{\em Syst\`emes dynamiques holomorphes}. In Bourbaki seminar, Vol. 1982/83,
volume 105 of Ast\'erisque, pp. 39--63. Soc. Math. France, 1983.


\bibitem[DoHu1]{DH1} {\sc A.~Douady}, {\sc J.~H.~Hubbard} ---
{\em \'Etude dynamique des polyn\^omes complexes I \&~II}, Publ.
math. d'Orsay (1984) \& (1985).

\bibitem[DoHu2]{DH2} {\sc A.~Douady}, {\sc J.~H.~Hubbard} ---
{\em On the dynamics of polynomial-like mappings}, Ann. scient.
\'Ec. Norm. Sup. {\bf 18} (1985), pp. 287--343.

\bibitem[Du]{Dudko} {\sc D. Dudko} ---
             {\em Matings with laminations  }, Arxiv:1112.4780 at arxiv.org.


\bibitem[Fa]{F} {\sc D.~Faught}, ---
             {\em Local connectivity in a family of cubic polynomials},
              Ph. D. Thesis, Cornell University, 1992.

\bibitem[GoMi]{GM} {\sc L.~R.~Goldberg}, {\sc J.~Milnor} ---
{\em Fixed Points of Polynomial maps. Part II. Fixed Point
Portraits}, Ann. sci. \'Ec. Norm. Sup.  {\bf 26} (1993), pp. 51--98.

\bibitem[Hu]{H} {\sc J.~H.~Hubbard} ---
{\em Local connectivity of Julia sets and bifurcation loci: three
theorems of J.-C. Yoccoz}, in  {\em Topological Methods in Modern
Mathematics}, pp. 467--511, Goldberg and Phillips eds, Publish or
Perish 1993.

\bibitem[LePr]{LP} {\sc G.~Levin}, {\sc F.~Przytycki} --
External rays to periodic points. Israel J. Math. {\bf 94} (1996),
pp. 29--57.

\bibitem[Ki]{Kiwi}{\sc J.~Kiwi}
{\em $\R$-laminations and the topological dynamics of complex polynomials}
Advances in Mathematics {\bf 184} (2004), pp. 207--267

\bibitem[Le]{Levin}{\sc G. Levin}, {\em On backward stability of holomorphic dynamical systems}, Fund. Math. {\bf 158} (1998), pp. 97--107.
\bibitem[MePe]{MP}{\sc D. Meyer, C. L. Petersen} ---{\em On The Notions of Mating},
Annales de la facult\' e des sciences de Toulouse S\' er. 6, {\bf 21} no. 5  (2012), pp. 839--876.
\bibitem[Mi1]{M1} {\sc J. Milnor}  ---
            {\em Dynamics in One Complex Variable},
             Vieweg 1999, 2nd edition 2000.
\bibitem[Mi3]{M3} {\sc J. Milnor}  ---
             {\em Local Connectivity of Julia Sets: Expository Lectures},
              pp. 67-116 of ``The Mandelbrot set, Theme and Variations''
              ed. Tan Lei, LMS Lecture Note Series {\bf 274} ,
              Cambridge U. Press 2000.
\bibitem[Mi4]{Mimat} {\sc J. Milnor} ---
{\em Pasting together Julia sets: a worked out example of mating}. Experiment.
Math., {\bf 13}, no. 1, (2004) pp. 55--92.

\bibitem[Mo]{Moore} {\sc R. L. Moore} --- {\em Concerning Upper Semi-
Continuous Collections of Continua.} Trans. Amer.
Math. Soc {\bf 27} (1925), pp. 416--428.

\bibitem[Na]{N} {\sc  V.A. Na\u\i shul'} ---
             {\em  Topological invariants of analytic and area-preserving
              mappings and their application to analytic differential
             equations in $ C\sp{2}$ and $ CP\sp{2}$. }
             Trans. Moscow Math. Soc. {\bf 42} (1983), pp. 239--250.

\bibitem[Pe]{Pe} {\sc C. L. Petersen}~---
{\em On the Pommerenke-Levin-Yoccoz inequality}, Ergod. Th. \&
Dynam. Sys., {\bf 13}, (1993), pp. 785--806.


\bibitem[Rees]{ReesV3} {\sc M. Rees} ---
{\em A fundamental domain for $V_3$}, M\'em. Soc. Math. Fr., (N.S.) {\bf 121}, (2010). 

\bibitem[Ro0]{Roesch-thesis} {\sc P. Roesch} ---
{\em Toplogie locale des m\'ethodes de Newton cubiques}, Th\`ese de l'ENS de Lyon, 1997.

\bibitem[Ro1]{Ro1} {\sc P. Roesch} ---
{\em Puzzles de Yoccoz pour les applications \`a allure
rationnelle}, L'Enseignement Math\'ematique,  {\bf 45}, (1999),
pp. 133--168.

\bibitem[Ro2]{Roeschcras} {\sc P. Roesch} ---
{\em Topologie locale des m\'ethodes de Newton cubiques:
plan param\'etrique}, C. R. Acad. Sci. Paris, {\bf t. 328}, S\'erie I,  (1999), pp. 151--154

\bibitem[Ro3]{RoeschENS} {\sc P.~Roesch}---
{\em Hyperbolic components of polynomials with a fixed critical point of maximal order},  Ann. scient.
\'Ec. Norm. Sup. {\bf 40} (2007), pp. 901--947.

\bibitem[Ro4]{RoeschAnnals} {\sc P.~Roesch} ---
{\em On local connectivity for the Julia set of rational maps: Newton's famous example.}
Annals of Mathematics, {\bf 168}, (2008), no. 1, pp. 127--174.

 \bibitem[RoYi]{RY}     {\sc P.~Roesch, Y.~Yin} ---
{\em
The boundary of bounded polynomial Fatou components}
 Comptes Rendus Acad. Sci. Paris S\'erie I {\bf 346}, Issue 15, (2008)
pp. 877--880.



\bibitem[Shi1]{Shi} {\sc M. Shishikura } ---
             {\em On a theorem of M. Rees for matings of polynomials},
             The Mandelbrot set, theme and variations, pp. 289--305, London Math. Soc. Lecture Note Ser., 274, Cambirdge University Press, Cambirdge 2000.

\bibitem[Shi2]{Shi2} {\sc M. Shishikura } ---
             {\em A family of cubic rational maps and matings of cubic polynomials},
             Experimental Math. {\bf 9}, (2000), no. 1, pp 29--53.


\bibitem[Ta1]{TL} {\sc L. Tan} ---
             {\em Branched coverings and cubic Newton maps},
             Fundamenta Mathematicae. {\bf 154}, (1997), no. 3, pp. 207--260.

\bibitem[Ta2]{TL2} {\sc L. Tan} ---
             {\em Matings of quadratic polynomials},
             Ergodic Theory and dynamical Systems. {\bf 12}, (1992), pp. 589--620.

\bibitem[Ti]{Timorin} {\sc V. Timorin} ---
             {\em The external boundary of $\mathcal M_2$ }, Fields Institute Communications Vol. 53, "Holomorphic Dynamics and Renormalization , A Volume in Honour of John Milnor's 75th Birthday", pp. 225--267.
              .
  \bibitem[YaZa]{YZ}{\sc  M.  Yampolsky, S. Zakeri} ---{\em  Mating Siegel quadratic polynomials}. J. Amer. Math. Soc. {\bf 14}, (2001), pp. 25--78.

  \end{thebibliography}
\end{document}